%% file: EffectiveDistTubes.tex
\documentclass[11pt]{amsart}

\usepackage{amsthm}
\usepackage{amsmath}
\usepackage{amssymb}
\usepackage{enumerate}
\usepackage{overpic}
\usepackage{color}
\usepackage[dvipsnames]{xcolor}
\usepackage{tikz-cd}
\usepackage{microtype}
\usepackage[hidelinks,pagebackref,pdftex]{hyperref}

\AtBeginDocument{%
   \def\MR#1{}
}

\usepackage{marginnote}
\long\def\@savemarbox#1#2{\global\setbox#1\vtop{\hsize\marginparwidth 
  \@parboxrestore\tiny\raggedright #2}}
\renewcommand*{\backref}[1]{}
\renewcommand*{\backrefalt}[4]{
  \ifcase #1
  [No citations.]
  \or [#2]
  \else [#2]
  \fi }

\numberwithin{equation}{section}
\theoremstyle{plain}
\newtheorem{theorem}[equation]{Theorem}

\newtheorem{lemma}[equation]{Lemma}

\newtheorem{proposition}[equation]{Proposition}

\newtheorem*{namedtheorem}{\theoremname}
\newcommand{\theoremname}{testing}

\theoremstyle{definition}
\newtheorem{definition}[equation]{Definition}
\newtheorem{remark}[equation]{Remark}

\newtheorem{example}[equation]{Example}

\newcommand{\from}{\colon} 

\newcommand{\HH}{{\mathbb{H}}}
\newcommand{\RR}{{\mathbb{R}}}
\newcommand{\ZZ}{{\mathbb{Z}}}
\newcommand{\NN}{{\mathbb{N}}}
\newcommand{\CC}{{\mathbb{C}}}

\newcommand{\PP}{{\mathbb{P}}}

\newcommand{\refthm}[1]{Theorem~\ref{Thm:#1}}
\newcommand{\reflem}[1]{Lemma~\ref{Lem:#1}}
\newcommand{\refprop}[1]{Proposition~\ref{Prop:#1}}

\newcommand{\refrem}[1]{Remark~\ref{Rem:#1}}

\newcommand{\refeqn}[1]{\eqref{Eqn:#1}}

\newcommand{\refdef}[1]{Definition~\ref{Def:#1}}
\newcommand{\refsec}[1]{Section~\ref{Sec:#1}}
\newcommand{\reffig}[1]{Figure~\ref{Fig:#1}}
\newcommand{\refexa}[1]{Example~\ref{Exa:#1}}

\newcommand{\bdy}{\partial}

\newcommand{\Real}{\operatorname{Re}}
\newcommand{\Imag}{\operatorname{Im}}

\newcommand{\vol}{\operatorname{vol}}

\newcommand{\area}{\operatorname{area}}

\newcommand{\smod}{{ \!\!\! \mod{} }}

\newcommand{\rmin}{{r_{\min}}}
\newcommand{\rmax}{{r_{\max}}}

\newcommand{\trad}{{\operatorname{trad}}}

\newcommand{\injrad}{{\operatorname{injrad}}}

\newcommand{\Hhat}{{\hat{\HH}^3}}
\newcommand{\Isom}{{\operatorname{Isom}}}

\DeclareMathOperator{\arccosh}{arccosh}
\DeclareMathOperator{\arcsinh}{arcsinh}

\renewcommand{\setminus}{-}

\title{Effective distance between nested Margulis tubes}

\author[D.~Futer]{David Futer}
\address[]{Department of Mathematics, Temple University,
Philadelphia, PA 19122, USA}
\email[]{dfuter@temple.edu}

\author[J.~Purcell]{Jessica S.~Purcell}
\address[]{School of Mathematical Sciences, Monash University, VIC 3800, Australia }
\email[]{jessica.purcell@monash.edu}

\author[S.~Schleimer]{Saul Schleimer}
\address[]{Department of Mathematics, 
University of Warwick, Coventry CV4 7AL, UK}
\email[]{s.schleimer@warwick.ac.uk}


\begin{document}

\begin{abstract}
We give sharp, effective bounds on the distance between tori of fixed injectivity radius inside a Margulis tube in a hyperbolic $3$--manifold. 
\end{abstract}

\maketitle

\section{Introduction}
\label{Sec:Intro}

A key tool in the study of hyperbolic manifolds is the \emph{thick-thin decomposition}.  For any number $\epsilon > 0$, a manifold $M$ is decomposed into the \emph{$\epsilon$--thin} part, consisting of points on essential loops of length less than $\epsilon$, and its complement the \emph{$\epsilon$--thick} part.  Margulis proved the foundational result that there is a universal constant $\epsilon_3 > 0$ such that for any hyperbolic $3$--manifold $M$, the  $\epsilon_3$--thin part is a disjoint union of cusps and tubes.  Analogous statements hold in all dimensions, and for more general symmetric spaces.  This result has had numerous important consequences in the study and classification of hyperbolic $3$--manifolds and Kleinian groups. 
Thurston and Jorgensen used the Margulis lemma to describe the structure of the set of volumes of hyperbolic manifolds, with limit points occurring only via Dehn filling~\cite{thurston:notes}. 
The Margulis lemma also plays a major role in the construction of model manifolds used in the proofs of the Ending Lamination Theorem~\cite{minsky:punctured-tori, brock-canary-minsky:elc} and the Density Theorem~\cite{namazi-souto:density, Ohshika05}.

Since tubes and cusps are well-understood quotients of hyperbolic space by elementary groups, it seems that the thin parts of manifolds should be easy to analyze.  However, in practice, the thin parts of a manifold are often very difficult to control. For example, the optimal value for the Margulis constant $\epsilon_3$ is still unknown.  The best known estimate is 
due to Meyerhoff~\cite{meyerhoff}.  Additionally, given an $\epsilon$--thin tube, it is very difficult to analyze and bound simple quantities such as the radius of the tube in full generality.  This is because the radius depends not only on $\epsilon$, but also on the rotation and translation --- the \emph{complex length} --- at the core of the tube.  Although the radius is a continuous function of these parameters, it is non-differentiable in many places.  See \refprop{TubeRadFormula} for a formula, and \reffig{Juarez} for a graph.

In this paper, we address the geometry of thin parts of (possibly singular) hyperbolic $3$--manifolds. Given $0 < \delta < \epsilon$ and a tube of injectivity radius $\epsilon/2$ (the $\epsilon$--tube, for short), we determine sharp, effective bounds on the distance between the boundaries of the $\epsilon$--tube and the $\delta$--tube. These bounds are independent of the complex length $\lambda + i\tau$ of the core of the tube. Our main result is  the following.

\begin{theorem}
\label{Thm:EffectiveDistTubes}
Suppose that $0 < \delta < \epsilon \leq 0.3$.  Let $N = N_{\alpha, \lambda, \tau}$ be a hyperbolic solid torus whose core geodesic has complex length $\lambda + i\tau$ and cone angle $\alpha \leq 2\pi$, where $\lambda \leq \delta$. 
Then the distance $d_{\alpha, \lambda, \tau}(\delta,\epsilon)$ between the $\delta$ and $\epsilon$ tubes satisfies
\[
\max \left\{ \frac{\epsilon - \delta}{2}, \, 
\arccosh \frac{\epsilon}{\sqrt{7.256 \, \delta}} - 0.0424 \right\}
\leq
d_{\alpha, \lambda, \tau}(\delta,\epsilon)
\leq
\arccosh \sqrt{ \frac{\cosh \epsilon - 1}{\cosh \delta - 1}  }.
\]
\end{theorem}

We remark that the argument of $\arccosh$ in the lower bound of \refthm{EffectiveDistTubes} may be less than $1$, hence $\arccosh(\cdot)$ is undefined. To remedy this, we employ the convention that an undefined value does not realize the maximum.  The real point is that the lower bound is not very strong (less than $\epsilon/2$) for any pair $(\delta, \epsilon)$ such that $\epsilon < \sqrt{ 7.256 \, \delta } $. On the other hand, the lower bound of \refthm{EffectiveDistTubes} is sharp up to additive error for any pair $( \delta, \epsilon)$ where $\epsilon \geq \sqrt{ 7.256 \, \delta } $. The upper bound is sharp for every pair $( \delta, \epsilon)$. See \refsec{TubeDistSharpness} 
for more details.

\subsection{Motivation and applications}
\label{Sec:Motivation}

Ineffective bounds on the distances between thin tubes were previously observed by Minsky \cite[Lemma 6.1]{minsky:punctured-tori} and Brock--Bromberg~\cite[Theorem~6.9]{brock-bromberg:density}, who credit the bound to Brooks and Matelski~\cite{brooks-matelski}. Universal bounds of this sort, depending only on $\delta$ and $\epsilon$, are required in the proof of the Ending Lamination Theorem, both for punctured tori \cite{minsky:punctured-tori} and for general surfaces \cite{minsky:models-bounds, brock-canary-minsky:elc}. In particular, Minsky used such bounds in the proof of the ``a priori bounds'' theorem \cite{minsky:models-bounds} that curves appearing in a hierarchy have universally bounded length. One consequence of ``a priori bounds'' is Brock's volume estimate for quasifuchsian manifolds and for mapping tori \cite{brock:fibered, brock:quasifuchsian}. A second consequence is the result (due to Minsky, Bowditch, and Brock--Bromberg \cite{bowditch:elc, brock-bromberg:inflexibility}) that  distance in the curve complex  of a surface $S$ is coarsely comparable to electric distance in a $3$--manifold of the form $S \times \RR$.

In a slightly different direction, Brock and Bromberg applied the ineffective bounds on distances between tubes to cone-manifolds, establishing uniform bilipschitz estimates between the thick part of a cusped $3$--manifold and the thick parts of its long Dehn fillings \cite{brock-bromberg:density}. This application requires a version of \refthm{EffectiveDistTubes} for solid tori with a cone-singularity at the core.
In turn, the Brock--Bromberg result has been combined with the Ending Lamination Theorem to show that geometrically finite hyperbolic $3$--manifolds are dense in the space of all (tame) hyperbolic $3$--manifolds~\cite{namazi-souto:density, Ohshika05}.

The past few years have witnessed an intense effort to make theorems in coarse geometry \emph{effective}, that is, to make the constants explicit. Recent effective results include, for instance, \cite{ATW:DistanceFormula, futer-schleimer:cusp-geometry, HPW:SlimUnicorns}. Finding an effective version of the distance between thin tubes has been a major obstacle to extending those efforts. \refthm{EffectiveDistTubes} provides such an effective result.

\refthm{EffectiveDistTubes} is already being applied to obtain effective versions of several results mentioned above. Futer and Taylor have outlined an effective ``a priori bounds'' theorem, combining \refthm{EffectiveDistTubes} with sweepout arguments \cite{futer-schleimer:cusp-geometry} and effective results about hierarchies \cite{ATW:DistanceFormula}. Aougab, Patel, and Taylor have found an effective ``electric distance'' theorem, again using a combination of \refthm{EffectiveDistTubes} and sweepout arguments. Finally, the authors of this paper have used \refthm{EffectiveDistTubes} in combination with a number of cone-manifold estimates due to Hodgson and Kerckhoff \cite{hk:univ, hk:shape} to effectivize the Brock--Bromberg bilipschitz theorem \cite{FPS:EffectiveBilipschitz}. Our effective results on cone-manifolds require \refthm{EffectiveDistTubes} to hold for singular solid tori.

Finally, \refthm{EffectiveDistTubes} offers a useful step toward finding the Margulis constant $\epsilon_3$. 
The current state of  knowledge is that $0.104 \leq \epsilon_3 \leq 0.775$, with the lower bound due to Meyerhoff \cite{meyerhoff} and the upper bound realized by the Weeks manifold.
  Furthermore, a theorem of Shalen \cite{shalen:margulis-numbers, Shalen:SmallOptimalMargulis}, building on earlier work of Culler and Shalen \cite{culler-shalen:margulis}, says that $0.29$ is a Margulis number for all but finitely many hyperbolic $3$--manifolds. That is, for all but finitely many choices of $M$, the $0.29$--thin part of $M$ is a disjoint union of cusps and tubes. Any manifold $M$ failing this property must be closed and must have $\vol(M) < 52.8$. By combining \refthm{EffectiveDistTubes} with our work on cone-manifolds, we produce an explicit lower bound on the injectivity radius of any exception to Shalen's theorem. 
  This makes it theoretically feasible (although computationally impractical) to enumerate all manifolds with $\vol(M) < 52.8$ and injectivity radius bounded below, and to determine their Margulis numbers  \cite{FPS:EffectiveBilipschitz}. 
  
%
%

\subsection{Distance between tubes, as a function}

Let $0 < \delta \leq \epsilon$ be injectivity radii, and consider a solid torus $N = N_{\alpha, \lambda, \tau}$ whose core geodesic has complex length $\lambda + i\tau$ and cone angle $\alpha \leq 2\pi$.  The distance between the $\epsilon$-- and $\delta$--tubes in $N$ is defined carefully in \refdef{TubeRad}.  We denote this distance by $d_{\alpha, \lambda,\tau}(\delta,\epsilon)$.  For our current discussion, it helps to note that $d_{\alpha, \lambda,\tau}(\delta,\epsilon)$ is the difference of tube radii of the $\epsilon$--tube and $\delta$--tube, and that each tube radius is determined by taking a maximum of many smooth functions.  See \refprop{TubeRadFormula} for an exact formula.  As a consequence, $d_{\alpha,\lambda,\tau}(\delta, \epsilon)$ is a continuous but only piecewise smooth function of the quantities $\delta$, $\epsilon$, $\lambda$, and $\tau$.

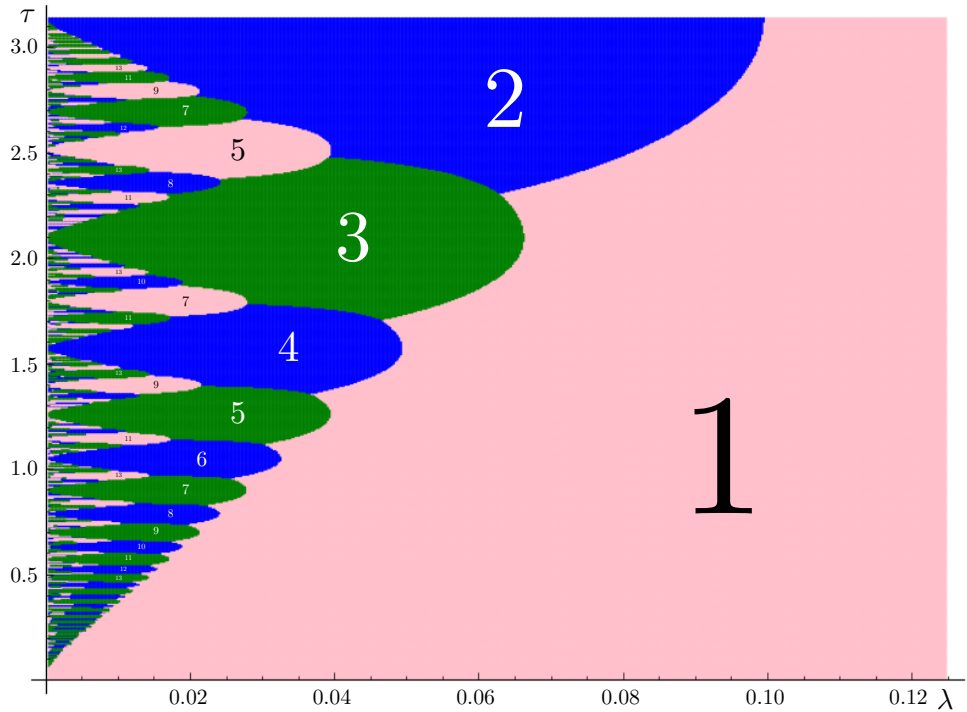
\begin{figure}[htbp]
\begin{overpic}[width=5in]{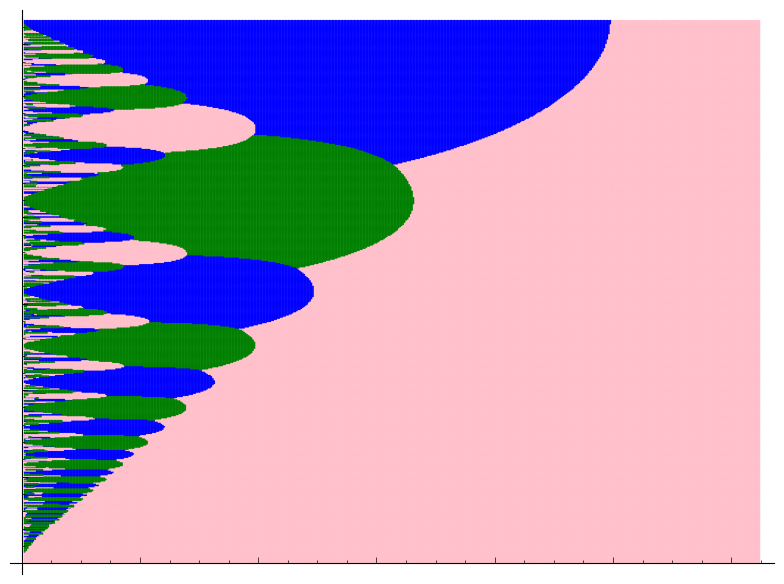}
\put(15.62,0.00){\scalebox{0.75}{$0.02$}}
\put(30.65,0.00){\scalebox{0.75}{$0.04$}}
\put(45.70,0.00){\scalebox{0.75}{$0.06$}}
\put(60.88,0.00){\scalebox{0.75}{$0.08$}}
\put(76.00,0.00){\scalebox{0.75}{$0.10$}}
\put(91.00,0.00){\scalebox{0.75}{$0.12$}}
\put(96,-0.5){$\lambda$}
\put(-1.0,12.80){\scalebox{0.75}{$0.5$}}
\put(-1.0,24.00){\scalebox{0.75}{$1.0$}}
\put(-1.0,34.80){\scalebox{0.75}{$1.5$}}
\put(-1.0,46.00){\scalebox{0.75}{$2.0$}}
\put(-1.0,57.00){\scalebox{0.75}{$2.5$}}
\put(-1.0,68.20){\scalebox{0.75}{$3.0$}}
\put(0,71.5){$\tau$}
\put(69.00,20.00){\scalebox{6.000}{\color{black}$1$}}
\put(48.50,60.50){\scalebox{3.000}{\color{white}$2$}}
\put(33.00,46.50){\scalebox{2.500}{\color{white}$3$}}
\put(27.00,36.00){\scalebox{1.500}{\color{white}$4$}}
\put(22.00,57.00){\scalebox{1.100}{\color{black}$5$}}
\put(22.00,29.50){\scalebox{1.100}{\color{white}$5$}}
\put(18.50,25.00){\scalebox{0.750}{\color{white}$6$}}
\put(17.00,61.75){\scalebox{0.500}{\color{white}$7$}}
\put(17.00,41.75){\scalebox{0.500}{\color{black}$7$}}
\put(17.00,22.00){\scalebox{0.500}{\color{white}$7$}}
\put(15.50,54.25){\scalebox{0.400}{\color{white}$8$}}
\put(15.50,19.65){\scalebox{0.400}{\color{white}$8$}}
\put(14.00,63.90){\scalebox{0.375}{\color{black}$9$}}
\put(14.00,33.10){\scalebox{0.375}{\color{black}$9$}}
\put(14.00,17.80){\scalebox{0.375}{\color{white}$9$}}
\put(12.30,44.00){\scalebox{0.275}{\color{white}$10$}}
\put(12.30,16.25){\scalebox{0.275}{\color{white}$10$}}
\put(11.00,65.40){\scalebox{0.250}{\color{white}$11$}}
\put(11.00,52.82){\scalebox{0.250}{\color{black}$11$}}
\put(11.00,40.30){\scalebox{0.250}{\color{white}$11$}}
\put(11.00,27.58){\scalebox{0.250}{\color{black}$11$}}
\put(11.00,15.08){\scalebox{0.250}{\color{white}$11$}}
\put(10.50,60.20){\scalebox{0.225}{\color{white}$12$}}
\put(10.50,13.95){\scalebox{0.225}{\color{white}$12$}}
\put(10.00,66.45){\scalebox{0.225}{\color{black}$13$}}
\put(10.00,55.73){\scalebox{0.225}{\color{white}$13$}}
\put(10.00,45.10){\scalebox{0.225}{\color{black}$13$}}
\put(10.00,34.50){\scalebox{0.225}{\color{white}$13$}}
\put(10.00,23.80){\scalebox{0.225}{\color{black}$13$}}
\put(10.00,13.07){\scalebox{0.225}{\color{white}$13$}}
\end{overpic}
\caption{In each region, all complex lengths $\lambda + i\tau$ have the indicated \emph{power} for $\epsilon = 0.2$.  That is: every based loop of length $\epsilon$ that realizes injectivity radius is freely homotopic to this power of the core.  This figure was inspired by~\cite[Figure~2]{CullerShalen:VolumeBetti2}. }
\label{Fig:EpsRadius}
\end{figure}

The failure of global smoothness can be explained as follows. For each value of $\epsilon > 0$ and each complex length $\lambda + i\tau$ of the core of $N$, the radius of the $\epsilon$--tube is determined by some power of the generator of $\pi_1(N)$.  This power can change as the data $\lambda,\tau,\epsilon$ changes. For instance, \reffig{EpsRadius} shows the power of the core when $\epsilon = 0.2$ is fixed but $\lambda + i\tau$ varies. Meanwhile, \refexa{Pi} shows how the power can depend on $\epsilon$.

\begin{figure}[t]
\centering 
\begin{overpic}[width=5in]{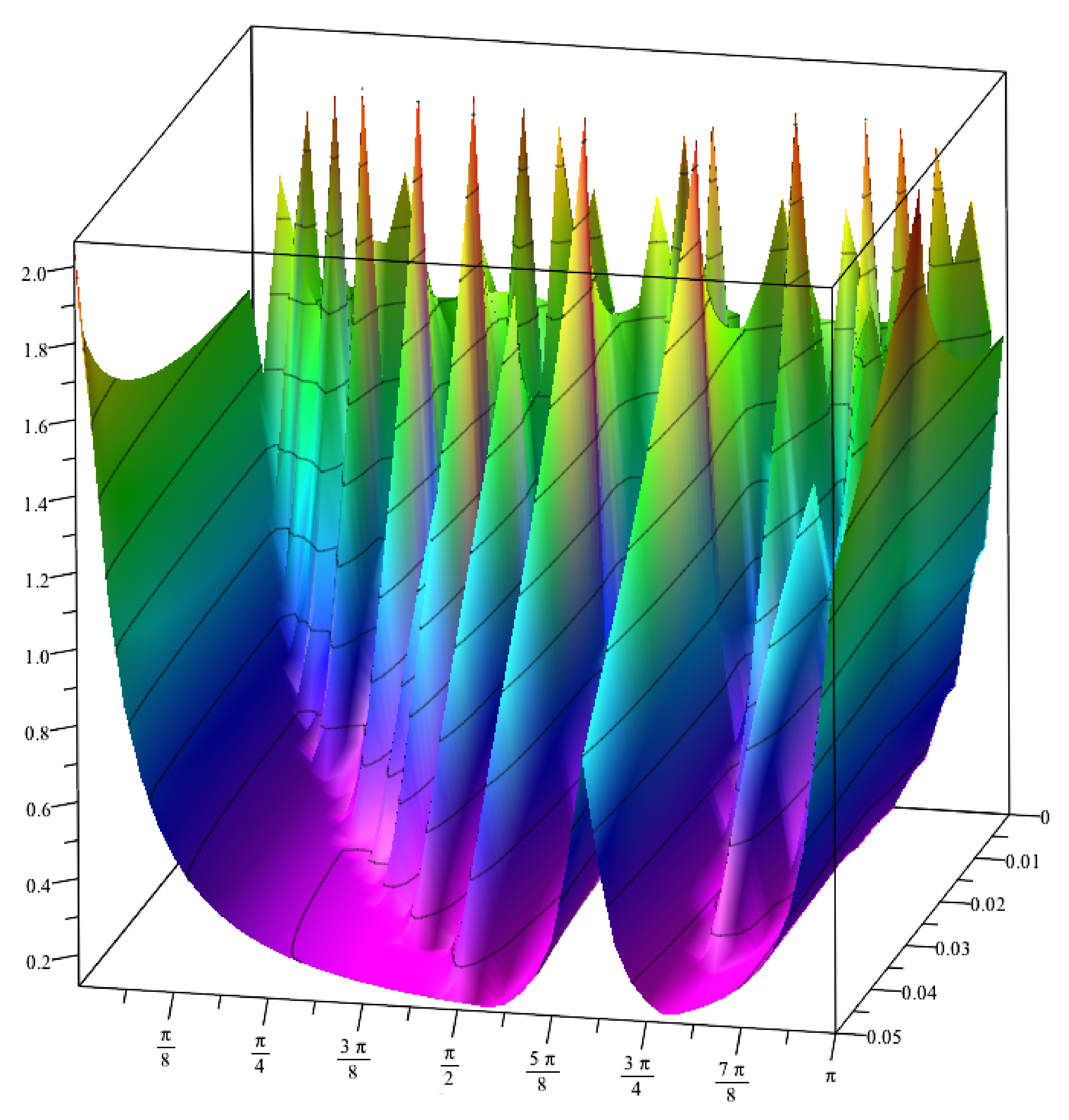}
\put(89,15){$\lambda$}
\put(36,2){$\tau$}
\end{overpic}
\caption{The graph of $d_{2\pi,\lambda, \tau}(\delta,\epsilon)$ when $\epsilon = 0.2$ and $\delta = 0.05$ are fixed, while $\lambda$ and $\tau$ are varying. The values of $d_{2\pi,\lambda, \tau}(\delta,\epsilon)$ range from approximately $0.117$ to $2.065$. By comparison, \refthm{EffectiveDistTubes} gives $\frac{\epsilon - \delta}{2} = 0.075$ as a fairly good lower bound and $2.0650 \ldots$ as a sharp upper bound.}
\label{Fig:Juarez}
\end{figure}

 \reffig{Juarez} shows the graph of $d_{2 \pi, \lambda, \tau}(\delta, \epsilon)$ when $\delta = 0.05$ and $\epsilon = 0.2$ are fixed but $(\lambda, \tau)$ vary.  Since the graph is the difference of a pair of wildly varying, piecewise-smooth functions, it is an extremely complicated terrain of deep valleys, narrow ridges, and sharp peaks.   
 The sharp ridges are points of non-differentiabilty, and occur where the power of the core for $\delta$ changes. Other points of non-differentiability, where the power for $\epsilon$ changes, occur in the valleys.
  Even though $\delta$ and $\epsilon$ are fixed, the value of  $d_{\alpha, \lambda,\tau}(\delta,\epsilon)$ ranges a great deal: from approximately $0.117$ to $2.065$. Nevertheless, \refthm{EffectiveDistTubes} gives upper and lower bounds that depend only on $\delta$ and $\epsilon$.

\subsection{Sharpness and numerical constants}
\label{Sec:TubeDistSharpness}

The hypothesis $\epsilon \leq 0.3$ in \refthm{EffectiveDistTubes} is slightly arbitrary. 
This hypothesis is not needed at all in the upper bound (see \refprop{EffectiveDistUpper}). In the lower bound,  our line of argument requires $\epsilon$ to be bounded in some way, with the choice of bound affecting the additive constant ($-0.0424$) in the statement. (See \refthm{EffectiveDistGeneral} for a generalized statement that holds for larger $\epsilon$.)
We chose the value $0.3$ because of its connection to currently available estimates on the Margulis constant $\epsilon_3$. In particular, as described in \refsec{Motivation}, applying \refthm{EffectiveDistTubes} with $\epsilon = 0.29$ can help determine the finite list of manifolds for which $0.29$ fails to be a Margulis number.

There are interesting examples illustrating the sharpness of both the upper and lower bounds of \refthm{EffectiveDistTubes}. As \refprop{EffectiveDistUpper} will show, the upper bound of \refthm{EffectiveDistTubes} is sharp for every pair $(\delta,\epsilon)$. It is attained if and only if $N$ is a nonsingular tube whose core has complex length $\lambda+i\tau = \delta + 0$; that is, the core has length $\delta$ and trivial rotational part.

The lower bound of \refthm{EffectiveDistTubes} is sharp up to additive error, which can be seen as follows. For every pair $(\delta, \epsilon)$ such that  $0 < \sqrt{7.256 \delta} \leq \epsilon \leq 0.3$, \refthm{Sharpness} constructs a solid torus $N = N_{2\pi, \lambda, \tau}$ such that
\begin{equation}
\label{Eqn:LowerBoundSharpness}
d_{2\pi, \lambda, \tau}(\delta, \epsilon) \: \leq \: 
\arccosh  \left( 1.116 \frac{\epsilon}{\sqrt{\delta} } \right).
\end{equation}
The core of $N$ has complex length $\lambda+i\tau = 1/n^2 + 2\pi i /n$, where $n$ is the least natural number such that $1/n^2 \leq \delta$. Meanwhile, \refthm{EffectiveDistTubes} gives
\begin{equation}
\label{Eqn:LowerBoundRestate}
d_{2\pi, \lambda, \tau}(\delta, \epsilon) \: \geq \: \arccosh \left( \frac{\epsilon}{\sqrt{7.256 \, \delta}} \right) - 0.0424.
\end{equation}
Since $\arccosh( x) \sim \log(2x)$ for large $x$, the expressions in~\refeqn{LowerBoundSharpness} and~\refeqn{LowerBoundRestate} 
differ by an additive error. In fact, the additive difference is less than $2.2$. 

One consequence of the above  paragraphs is that \refthm{EffectiveDistTubes} is sharpest when the solid torus $N$ is nonsingular; that is, when $N$ is the quotient of $\HH^3$ by a loxodromic isometry. Thus, while extending \refthm{EffectiveDistTubes} to singular tubes introduces a few technical complications (for example see Propositions~\ref{Prop:TubeRadFormula} and~\ref{Prop:MultiplicativeRadGap}), this extension does not weaken the statement in any way. 
We emphasize that the extension to singular tubes is needed for our forthcoming applications to cone-manifolds and to bounding the Margulis numbers of (nonsingular) hyperbolic manifolds.

   \smallskip

The  results of this paper have an interesting relation to the discussion of Margulis tubes in the work of Minsky (see \cite[Section 6]{minsky:punctured-tori} and \cite[Section 3.2.2]{minsky:models-bounds}). On the one hand, \refprop{TubeRadBound} confirms and effectivizes Minsky's assertion (\cite[proof of Lemma~6.1]{minsky:punctured-tori} and \cite[Equation~(3.6)]{minsky:models-bounds}) that, for $\epsilon \leq \epsilon_3$, there is a constant $c = c(\epsilon)$ such that
the radius of an $\epsilon$--tube of core length $\lambda + i \tau$ satisfies
\[
r_{2\pi,\lambda, \tau}(\epsilon) \geq \log \left( \frac{1}{\sqrt{\lambda}} \right) -c .
\]
On the other hand, the examples of \refthm{Sharpness} satisfy~\refeqn{LowerBoundSharpness} and therefore contradict  \cite[Equation~(3.7)]{minsky:models-bounds}. Of course, this does not affect the overall correctness of \cite{minsky:models-bounds}, as \emph{any} lower bound depending only on $\epsilon$ and $\delta$ that grows large as $\delta \to 0$ (for instance, the lower bound of \refthm{EffectiveDistTubes}) suffices for Minsky's work toward the classification of Kleinian surface groups.

\subsection{Cusps}
Recall that the thin part of a hyperbolic 3--manifold is a disjoint union of cusps and tubes. Although \refthm{EffectiveDistTubes} is only stated for tubes, there is a simpler and stronger statement for cusps.

\begin{proposition}
\label{Prop:EffectiveDistCusps}
Let $0 < \delta < \epsilon$.  Let $N$ be a horocusp whose $\epsilon$--thick part, $N^{\geq \epsilon}$, is non-empty.
Then the $\delta$--thin and $\epsilon$--thick parts of $N$ are separated by distance
\[
d_N(\delta, \epsilon) = d(N^{\leq \delta}, N^{\geq \epsilon}) = \log \left( \frac{ \sinh(\epsilon/2) }{ \sinh (\delta/2) } \right) .
\]
\end{proposition}

\begin{proof}
Let $x \in N$ be a point such that $\injrad(x) = \delta / 2$. Then
there must be a lift $\widetilde x \in \widetilde N$ and a parabolic covering transformation $\varphi$, such that $d(\widetilde x, \varphi \widetilde x) = \delta$. By \cite[Lemma A.2]{cfp:tunnels}, a horocyclic segment from $\widetilde x$ to $\varphi  \widetilde x$ has length $2 \sinh( \delta / 2)$. Similarly, if $y \in N$ is a point such that $\injrad(y) = \epsilon / 2$, then there is a  horocyclic segment from $\widetilde y$ to $\varphi \widetilde y$ of length $2 \sinh( \epsilon / 2)$. Now, a standard calculation shows that the horospheres containing $\widetilde x$ and $\widetilde y$ are separated by distance 
\[
d_N(\delta, \epsilon) = d(N^{\leq \delta}, N^{\geq \epsilon}) = \log \left( \frac{2 \sinh(\epsilon/2) }{2 \sinh (\delta/2) } \right) . \qedhere
\]
\end{proof}

We observe that the distance $d_N(\delta, \epsilon)$ satisfies both the upper and lower bounds of \refthm{EffectiveDistTubes}. Thus \refthm{EffectiveDistTubes} also applies to cusps.

\subsection{Organization}

\refsec{Equidistant} lays out definitions and sets up notation that will be used for the remainder of the paper. \refsec{TubeRad} proves \refprop{TubeRadFormula}, which gives an exact formula for the radius of an $\epsilon$--thin tube.

\refsec{Sharpness} describes a family of examples showing the sharpness of the lower bound of \refthm{EffectiveDistTubes}.  \refsec{Upper} proves the upper bound of \refthm{EffectiveDistTubes} and shows that it is sharp.

The lower bound of \refthm{EffectiveDistTubes} requires a delicate case analysis, treating shallow and deep tubes separately. For a maximally sharp estimate, we rely on a result of 
Zagier~\cite{meyerhoff}, later improved by Cao, Gehring, and Martin~\cite{CaoGehringMartin}. 
 We obtain a bound on the Euclidean metric on the tube boundary in \refsec{Euclidean}, and prove a lower bound on the depth of an $\epsilon$--tube in \refsec{TubeDepth}. These ingredients are combined to give the final proof in \refsec{Lower}.

Appendix~\ref{Sec:Trig} contains several elementary lemmas in hyperbolic trigonometry that are useful elsewhere in the paper.

\subsection{Acknowledgements}
Futer was  supported in part by NSF grant DMS--1408682. Purcell was supported in part by the Australian Research Council.  All three authors acknowledge support from NSF grants DMS--1107452, 1107263, 1107367, ``RNMS: Geometric Structures and Representation Varieties'' (the GEAR Network), which funded an international trip to collaborate on this paper.

We thank Ian Biringer and Yair Minsky for a number of enlightening conversations.  We also thank Tarik Aougab, Marc Culler, Priyam Patel, and Sam Taylor for their helpful suggestions. 

\section{Tubes and equidistant tori}
\label{Sec:Equidistant}

In this section, we set notation and give definitions for solid tori, tubes, and injectivity radii that will be used for the remainder of the paper.

To form a nonsingular hyperbolic tube, one starts with a neighborhood of a geodesic in $\HH^3$ and takes a quotient under a loxodromic isometry fixing that geodesic. In order for our results to hold for both nonsingular and singular tubes, we will take a quotient of a more complicated space, as in the following definition.  Fix $0 \in \HH^3$ to be an arbitrary basepoint. 

\begin{definition}
\label{Def:SingularModel}
Let $\sigma \subset \HH^3$ be a bi-infinite geodesic. Let $\Hhat$ denote the metric completion of the universal cover of $(\HH^3 \setminus \sigma)$. Let $\hat{\sigma}$ be the set of points added in the completion.

The space $\Hhat$ can be regarded as an infinite cyclic branched cover of $\HH^3$, branched over $\sigma$. The branch set $\hat{\sigma} \subset \Hhat$ is a singular geodesic with infinite cone angle. 

There is a natural action of $\CC$ (as an additive group) on $\Hhat$, where $z \in \CC$ translates $\hat{\sigma}$ by distance $\Real(z)$ and rotates by angle $\Imag(z)$. Since $\hat{\sigma}$ has infinite cone angle, we have that angles of rotation are real-valued. Conversely, every isometry $\varphi$ of $\Hhat$ that preserves orientation on both $\Hhat$ and $\hat{\sigma}$ comes from this action, and has a well-defined \emph{complex length} $z = \zeta+i\theta$. We can therefore write $\varphi = \varphi_{\zeta+i\theta}$.

We endow $\Hhat$ with a system of cylindrical coordinates $(r, \zeta, \theta)$, as follows. Choose a reference ray perpendicular to $\hat{\sigma}$, and let the points of this ray have coordinates $(r,0,0)$, where $r \geq 0$ measures distance from $\hat{\sigma}$. Then, let $(r,\zeta,\theta)$ be the image of $(r,0,0)$ under the isometry $\varphi_{\zeta+i\theta}$. The distance element in these coordinates is $ds$, where
\begin{equation}
\label{Eqn:CylindricalCoords}
ds^2 = dr^2 + \cosh^2 r \, d\zeta^2 + \sinh^2 r \, d\theta^2 .
\end{equation}
\end{definition}

\begin{definition}
\label{Def:ModelSolidTorus}
Consider a group $G = \ZZ \times \ZZ$ of isometries of $\Hhat$, generated by an elliptic $\psi_{i\alpha}$ and a loxodromic $\varphi = \varphi_{\lambda+i\tau}$, where $\alpha > 0$ and $\lambda > 0$. The quotient space $N_{\alpha,\lambda,\tau}$ is an open solid torus whose core curve $\Sigma$ is a closed geodesic of complex length $\lambda + i\tau$, and with a cone singularity of angle $\alpha$ at the core. We call $N = N_{\alpha,\lambda,\tau}$ a \emph{model solid torus}.

Note that if $\alpha = 2\pi$, the quotient of $\Hhat$ by the elliptic $\psi_{i\alpha}$ recovers $\HH^3$. In this case, the model solid torus $N_{2\pi, \lambda, \tau}$ is \emph{nonsingular}, and can be identified with $\HH^3/ \langle \varphi_{\lambda+i \tau} \rangle$, where $\varphi_{\lambda+i\tau}$ is a loxodromic isometry of $\HH^3$ with complex length $\lambda + i \tau$.
\end{definition}

\begin{definition}
\label{Def:Injectivity}
Let $N = N_{\alpha, \lambda, \tau}$ be a  model solid torus, and let $x \in N$. Then the \emph{injectivity radius} of $x$, denoted $\injrad(x)$, is the supremal radius $r$ such that an open metric $r$--ball about $x$ is isometric to a ball $B_r(0) \subset \HH^3$. Since we are using open balls, the supremal radius is attained. If $\alpha \neq 2\pi$ and $x$ lies on the singular core geodesic, we set $\injrad(x) = 0$.

For $\epsilon > 0$, the \emph{$\epsilon$--thick part} of $N$ is 
\[
N^{\geq \epsilon} = \{ x \in N : \injrad(x) \geq \epsilon / 2 \} .
\]
The \emph{$\epsilon$--thin part} is $N^{< \epsilon} = N \setminus N^{\geq \epsilon}$.
We define $N^{\leq \epsilon}$ and $N^{> \epsilon}$ similarly.
\end{definition}

We emphasize that our definition of the \emph{$\epsilon$--thick part} corresponds to injectivity radius $\epsilon/2$ rather than $\epsilon$.
 Both choices seem to be common in the literature on Kleinian groups. Our convention agrees with that of Minsky~\cite{minsky:punctured-tori, minsky:models-bounds} and Brock--Canary--Minsky~\cite{brock-canary-minsky:elc}, while differing from the convention of Brock--Bromberg \cite{brock-bromberg:density} and Namazi--Souto \cite{namazi-souto:density}. 

If $M$ is a non-singular manifold, every essential loop through a point $x \in M^{\geq \epsilon}$ has length at least $\epsilon$. A similar statement holds for singular tubes.


\begin{lemma}
\label{Lem:InjRadLoop}
Let $N = N_{\alpha, \lambda, \tau}$ be a model solid torus whose core is singular. That is, assume $N$ is the quotient of $\Hhat$ by $\ZZ^2 \cong \langle \psi_{i\alpha}, \varphi_{\lambda+i\tau} \rangle$ for $\alpha \neq 2\pi$. Choose a point $x \in N$ and a lift $\widetilde x \in \Hhat$. Set $\epsilon = 2\, \injrad(x)$. Then 
\begin{align}
\label{Eqn:InjectivityTrans}
\epsilon
& =  \min \left\{ d( \widetilde x, \eta \widetilde x)
 :  \eta \in \ZZ^2 \setminus \{ 0 \}  \right\}. 
\end{align}
Similarly, if $N$ is a nonsingular solid torus, the quotient of $\HH^3$ by
$\ZZ \cong \langle \varphi_{\lambda+i\tau}\rangle$, and $\widetilde x \in \HH^3$ is a point covering $x$, then 
\begin{align}
\label{Eqn:NonsingInjectivityTrans}
\epsilon
& =  \min \left\{ d( \widetilde x, \eta \widetilde x)
 :  \eta = \varphi_{\lambda+i\tau}^n \in \ZZ \setminus \{ 0 \}  \right\}. 
 \end{align}
\end{lemma}

\begin{proof}
We focus on the singular case, as the nonsingular case is well-known. 
For an arbitrary basepoint $0 \in \HH^3$, there is an isometric embedding $f \from B_{\epsilon/2}(0) \to N $, such that $f(0) = x$. Then any non-trivial element $\eta \in \ZZ^2$ must translate $B_{\epsilon/2}(\widetilde x)$ by distance at least $\epsilon$, so
\[
\min \{ d(\widetilde x, \eta\widetilde x) : \eta\in \ZZ^2\setminus\{0\} \} \geq \epsilon. 
\]

Next, since $\injrad(x) = \epsilon/2$, the continuous extension of $f$ to $\overline{B_{\epsilon/2}(0)}$ either hits the core $\Sigma$ or fails to be one to one. If it fails to be one to one, then a lifted ball is tangent to a translate of itself under a nontrivial element $\eta \in \ZZ^2$. On the other hand, if it meets the core $\Sigma$, then there is a point $z \in f(\overline{B_{\epsilon/2}(0)}) \cap \Sigma$. The preimage $\widetilde z$ of $z$ is fixed by an elliptic subgroup $\langle \psi_{i \alpha} \rangle \subset \ZZ^2$. Thus a translate of the lifted ball by the generator $\eta = \psi_{i \alpha}$ will be tangent to the ball. 

In either case, there are two distinct lifts of the ball, namely $B_{\epsilon/2}(\widetilde x)$ and $B_{\epsilon/2}(\eta \widetilde x)$, that are tangent in $\Hhat$.  Therefore, $d( \widetilde x, \eta \widetilde x) = \epsilon$, and the minimum over all group elements must be at most $\epsilon$.  The minimum is attained because $\ZZ^2$ acts discretely.
\end{proof}

\begin{definition}
\label{Def:Power}
If $N$ is a nonsingular tube and $\epsilon \geq \lambda$, we define the \emph{power} for $\epsilon$ to be any $n \in \NN$ so that the deck transformation $\eta = \varphi^n$ realizes the minimum in Equation~\refeqn{NonsingInjectivityTrans}.  If $N$ is a singular tube and $\epsilon > 0$, we define the \emph{power} for $\epsilon$ to be any $n \in \NN \cup \{ 0 \}$ so that the deck transformation $\eta = \varphi^n \psi^m$ realizes the minimum in Equation~\refeqn{InjectivityTrans}, for some $m \in \ZZ$.  The power is uniquely defined for almost every $\epsilon$.
\end{definition}

We remark that the power for $\epsilon$ is genuinely a function of $\epsilon$ and $N$, and does not depend on the choice of $x \in N$ such that $\injrad(x) = \epsilon/2$. 
This is because for every $\epsilon > 0$ and every loxodromic $\eta \in \Isom(\Hhat)$, the set $\{y \in \Hhat : d(y, \eta y) = \epsilon \}$ is a Euclidean plane at fixed radius from the core geodesic $\hat \sigma$ (or equal to $\hat \sigma$, or empty). The quotient of this set in $N$ is equidistant from the core of $N$. This is a distinctly $3$--dimensional phenomenon: already for $\eta \in \Isom(\HH^4)$, the set of points moved by distance $\epsilon$ can be far more complicated \cite{Susskind}.

\section{Tube radii}
\label{Sec:TubeRad}

We now provide a bit of background for tube radii in hyperbolic $3$--manifolds, well-known to the experts.  Let $N = N_{\alpha, \lambda, \tau}$ be a model solid torus, as in \refdef{ModelSolidTorus}.  Our eventual goal is to bound distances between the boundaries of the $\epsilon$--thin and $\delta$--thin tubes $N^{\leq\epsilon}$ and $N^{\leq\delta}$, independently of $\alpha$, $\lambda$, and $\tau$.  In this section, we derive a formula for the radius of one such tube. 

\begin{definition}
\label{Def:TubeRad}
Let $\epsilon > 0$, and assume the $\epsilon$--thin part $N^{< \epsilon}$ is non-empty. Then $T^\epsilon = \bdy N^{\leq\epsilon}$ is a torus consisting of points whose injectivity radius is exactly $\epsilon/2$. All of the points of $T^\epsilon$ lie at the same radius from the core geodesic $\Sigma$. We denote this radius by
\[
r(\epsilon) = r_{\alpha, \lambda, \tau} (\epsilon).
\]
We let $T_r$ denote the equidistant torus at radius $r$ from the core of $N$. Subscripts denote radius, while superscripts denote thickness.  Thus
\[
T^\epsilon = T_{r(\epsilon)}.
\]
Given $0 < \delta < \epsilon$, where $N^{< \delta} \neq \emptyset$, we define the distance $d(\delta, \epsilon)$ between $T^\delta$ and $ T^\epsilon$ as follows:
\begin{equation}
\label{Eqn:DistDef}
d(\delta,\epsilon) = d_{\alpha,\lambda,\tau}(\delta,\epsilon) = r_{\alpha, \lambda, \tau} (\epsilon) - r_{\alpha, \lambda, \tau} (\delta).
\end{equation}
\end{definition}

Our next goal is to find a formula for $r(\epsilon)=r_{\alpha,\lambda,\tau}(\epsilon)$ in terms of complex lengths of isometries stabilizing the singular geodesic of $\Hhat$.

\begin{definition}
\label{Def:TranslationRad}
Let $\varphi_{\lambda+i\tau}$ be an isometry of $\Hhat$, fixing the singular geodesic $\hat \sigma$, with complex length $\lambda+i\tau$.
For $\epsilon \geq \lambda$, define the \emph{translation radius}, denoted $\trad_{\lambda,\tau}(\epsilon)$, to be the value of $r$ such that $\varphi_{\lambda+i\tau}$ translates all points of the form $(r,\zeta,\theta) \in \Hhat$ by distance $\epsilon$.

We also need to compute the translation radius of a loxodromic isometry $\varphi_{\lambda+i \tau}$ acting on $\HH^3$ instead of $\Hhat$. Since angles in $\HH^3$ are only defined modulo $2\pi$, this radius coincides with the translation radius of $\varphi_{\lambda+i (\tau \smod 2\pi)}$ acting on $\Hhat$, namely    $\trad_{\lambda, \, \tau \smod 2\pi}(\epsilon)$. \end{definition}

The usage $(\tau \smod 2\pi)$ can be generalized to angles modulo other numbers.

\begin{definition}
\label{Def:Mod}
 Given $a \in \RR$ and $b > 0$, we define $(a \, \smod b)$ to be 
\[
a \smod b = x \in [-b/2, \, b/2) \quad \text{such that} \quad (a - x) \in b \ZZ.
\]
\end{definition}

In many situations, the translation radius can be computed in closed form.

\begin{lemma}
\label{Lem:TubeTransDistance} 
Let $(r, \zeta_1, \theta_1)$ and $(r, \zeta_2, \theta_2)$ be points of $\Hhat$ in cylindrical coordinates, and let $d$ be the distance between those points. If $|\theta_1 - \theta_2| \leq \pi$, then
\begin{align}
\cosh d &= \cosh (\zeta_1 - \zeta_2) \cosh^2 r - \cos (\theta_1 - \theta_2) \sinh^2 r \nonumber \\
             &= (\cosh (\zeta_1 - \zeta_2) - \cos (\theta_1 - \theta_2)) \cosh^2 r + \cos (\theta_1 - \theta_2). \label{Eqn:GMMDistance}
\end{align}
If $|\theta_1 - \theta_2| \geq \pi$, then 
\[
\cosh d = (\cosh (\zeta_1 - \zeta_2) + 1) \cosh^2 r) -1.
\]
Consequently, when $\epsilon \geq \lambda$ and $0 \leq |\tau | \leq \pi$, we have
\begin{equation}
\label{Eqn:TradDistance}
r = \trad_{\lambda,\tau}(\epsilon) = \arccosh \sqrt{ \frac{ \cosh \epsilon - \cos (\tau)}{ \cosh \lambda - \cos (\tau)}.    }
\end{equation}
\end{lemma}

\begin{proof}
Equation~\refeqn{GMMDistance} is proved in Gabai--Meyerhoff--Milley \cite[Lemma~2.1]{GabaiMeyerhoffMilley:VolTubes}. See Hodgson--Kerckhoff~\cite[Lemma~4.2]{hk:univ} for the extension to singular tubes in the case where $|\theta_1 - \theta_2| \leq \pi$. If $|\theta_1 - \theta_2| \geq \pi$, then the geodesic in $\Hhat$ between the two points will pass through the singular axis $\hat \sigma$. Thus the distance remains the same if we set $|\theta_1 - \theta_2| = \pi$, hence $\cos(\theta_1 - \theta_2) = -1$.

For Equation~\refeqn{TradDistance}, we substitute $d = \epsilon$ and  $(\zeta_1, \theta_1) = (\lambda, \tau)$ and $(\zeta_2, \theta_2) = (0, 0)$. Solving for $r$ in~\refeqn{GMMDistance} gives the result.
\end{proof}

\begin{remark}
\label{Rem:WeirdArccosh}
It will be convenient to refer to \refdef{TranslationRad} and \reflem{TubeTransDistance} even in situations when  $\epsilon < \lambda$, hence when there \emph{do not exist} points of $\Hhat$ that $\varphi_{\lambda+ i \tau}$ translates by distance $\epsilon$. If $\epsilon < \lambda$, we define   $\trad_{\lambda,\tau}(\epsilon) = -\infty$. Similarly, when $x < 1$, we set $\arccosh(x) = - \infty$.
Under this convention,~\refeqn{TradDistance} holds for all pairs $(\lambda, \epsilon)$. 
\end{remark}

The translation radius $\trad_{\lambda,\tau}(\epsilon)$ is related to the radius of the $\epsilon$--thin part of $N = N_{\alpha, \lambda, \epsilon} $, but they are not necessarily identical.

\begin{figure}
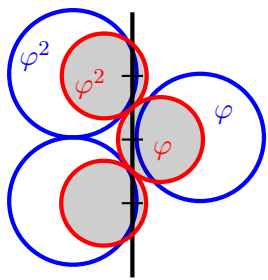
\caption{Schematic showing how the tangency pattern of balls $B_\epsilon$, $\varphi(B_\epsilon)$, and $\varphi^2(B_\epsilon)$ depends on $\epsilon$. In this figure, $\alpha=2\pi$, $\lambda=0.1$, $\tau=\pi$ are fixed. When $\epsilon$ is small, the balls $B_\epsilon$ and $\varphi(B_\epsilon)$ are tangent first (shaded). When $\epsilon$ is larger, $B_\epsilon$ and $\varphi^2(B_\epsilon)$ become tangent first. }
\label{Fig:Example}
\end{figure}

\begin{example}
\label{Exa:Pi}
Set $\alpha=2\pi$, $\lambda=0.1$, and $\tau=\pi$. Then $N= N_{\alpha, \lambda, \tau}$ is a quotient of $\HH^3$.
The generator $\varphi = \varphi_{\lambda, \tau}$ of $\pi_1 N \cong \ZZ$ will translate any point $\widetilde x \subset \HH^3$  along the invariant geodesic $\sigma$, and also rotate it by $\pi$ about $\sigma$. When $\epsilon < 0.1$,  we have $N^{\leq \epsilon} = \emptyset$, or equivalently $\trad_{\lambda,\tau}(\epsilon) = - \infty$. When $0.1 \leq \epsilon \leq 0.2$, the map $\varphi^n$ for $n\geq 2$ will translate any point of $\HH^3$, including points of $\sigma$, by a distance larger than $\epsilon$.
Thus, the radius of $N^{\leq \epsilon}$ is governed by $\varphi$ alone: that is, $r(\epsilon) = \trad_{\lambda, \tau}(\epsilon)$ and the power for $\epsilon$ is $1$.

Now, consider what happens when $\epsilon$ is ever so slightly larger than $0.2$, say $\epsilon = 0.201$. Because $\varphi^2$ has trivial rotational part, a point $\widetilde x$ can lie relatively far from $\sigma$ (at radius $r = 0.1001 \ldots$) and still be translated by distance $\epsilon$. In symbols, for a point $\widetilde x = (r, 0, 0) \in \HH^3$, we have
\[
d(\widetilde x, \varphi^2 \widetilde x) = \epsilon
\qquad \text{and} \qquad
\trad_{2\lambda, \, 2\tau \!\smod 2\pi}(\epsilon) = \trad_{2\lambda, \, 0 }(\epsilon) =  r = 0.1001 \ldots.
\]
Meanwhile, because $\varphi$ rotates points by angle $\pi$ about $\sigma$, it will move the same point $\widetilde x = (r, 0, 0) \in \HH^3$ by a distance much larger than $\epsilon$, whereas the points translated by distance $\epsilon$ lie closer to $\sigma$:
\[
d(\widetilde x, \varphi \widetilde x) = 0.2239 \ldots > \epsilon
\qquad \text{and} \qquad
\trad_{\lambda,\tau}(\epsilon) = 0.0871 \ldots < r.
\]
Finally, $\varphi^n$ for $n > 2$ has both translational and rotational part larger than that of $\varphi^2$, so $\varphi^n$ will definitely move any point of $\HH^3$ further than $\varphi^2$. In fact, for every $\epsilon \geq 0.201$, the radius of $N^{\leq \epsilon}$ is governed by $\varphi^2$: that is,  $r(\epsilon) = \trad_{2\lambda, 2\tau}(\epsilon)$ and the power for $\epsilon$ is $2$. See \reffig{Example}.
\end{example}

The above example is instructive in two ways. First, it illustrates that the power for $\epsilon$ is locally constant but jumps when $\epsilon$ crosses certain isolated values. Additional examples of this phenomenon are described in  \refsec{Sharpness}.  Second, the displayed equations above show that while $\epsilon = 2\injrad(x)$ is determined by taking a \emph{minimum} distance over all nonzero powers of $\varphi$ (see \reflem{InjRadLoop}), the tube radius $r(\epsilon)$ is determined by taking a \emph{maximum} value of $\trad_{n\lambda,n\tau}(\epsilon)$ over all nonzero powers. \refprop{TubeRadFormula} makes this idea precise, for singular as well as nonsingular tubes.

Before moving on, we note that the power for $\epsilon$ can only jump upward as $\epsilon$ increases; see \refrem{PowerMonotonicity}. When $\epsilon$ is fixed but $\lambda + i \tau$ varies, the jumping behavior of powers follows the combinatorics of the Farey graph, as illustrated in \reffig{EpsRadius}. 
The jumping behavior also explains why the function $r_{\alpha,\lambda,\tau}(\epsilon)$ and the related function $d_{\alpha,\lambda,\tau}(\delta, \epsilon)$ have points of non-differentiability, as visible in \reffig{Juarez}.

We are now ready to give a formula for the radius of the $\epsilon$--tube $N^{\leq \epsilon}$. The following is the main result of this section. 

\begin{proposition}
\label{Prop:TubeRadFormula}
Let $N = N_{\alpha, \lambda, \tau}$ be a model solid torus of cone angle $\alpha \leq 2\pi$. For any $\epsilon$ such that $N^{\leq \epsilon} \neq \emptyset$, the radius $r(\epsilon) = r_{\alpha,\lambda,\tau}(\epsilon)$ of the tube $N^{\leq \epsilon}$ can be computed as follows.

When the tube is nonsingular (that is, $\alpha = 2\pi$)
\begin{equation*}
r_{2\pi, \lambda, \tau}(\epsilon) = \max_{n \in \NN}  \big \{ \trad_{n\lambda, \, n\tau \!\smod 2\pi}(\epsilon) \big \}
=  \max_{n \in \NN}  \left \{    \arccosh \sqrt{ \frac{ \cosh \epsilon - \cos (n \tau)}{ \cosh (n\lambda) - \cos (n\tau)}    }  \right \} .
\end{equation*}

When the tube is singular (that is, $0 < \alpha < 2\pi$)
\begin{equation*}
  r_{\alpha,\lambda,\tau}(\epsilon) =\max\left\{ \trad_{0,\alpha}(\epsilon),\,
    \max_{n \in \NN}  \left\{ \trad_{n\lambda, \, n\tau \!\smod \alpha}(\epsilon) \right\}\right\}. 
\end{equation*}
Furthermore, if $\pi \leq \alpha <2\pi$, then $\trad_{0,\alpha}(\epsilon) = \epsilon/2.$
%
\end{proposition}

\begin{remark}
\label{Rem:TradMax}
In each case of \refprop{TubeRadFormula}, it suffices to take a maximum over a finite set. Indeed, when $n\lambda > \epsilon$, the translation length of $\varphi^n$ is sure to be larger than $\epsilon$.  Thus, by \refrem{WeirdArccosh}, all values of $n$ larger than $\epsilon/\lambda$ will contribute $- \infty$ to the set over which we are taking a maximum, and can therefore be ignored. 

The integer $n \in \NN \cup \{ 0 \}$ that realizes the maximum is the same as the power for $\epsilon$, defined in \refdef{Power}.
This shows that $r(\epsilon)$ is actually a maximum rather than a supremum.

\end{remark}

\begin{proof}[Proof of \refprop{TubeRadFormula}]
First, assume that $N$ is nonsingular. In this case, the solid torus $N$ can be described as $N_{2\pi, \lambda, \tau} = \HH^3 / \langle \varphi_{\lambda+ i \tau} \rangle \cong \HH^3 / \ZZ $. Let $x \in N$ be a point such that $2\injrad(x) = \epsilon$, and let $\widetilde x$ be a preimage in $\HH^3$. By \reflem{InjRadLoop}, we have
\[
\epsilon = \min \left\{ d( \widetilde x, \varphi^n \widetilde x)
 :   n \in \ZZ \setminus \{ 0 \}  \right\} .
\]
Let $m \in \ZZ \setminus \{ 0 \}$ be a power realizing the minimum. Without loss of generality, we may assume $m > 0$. Then, for every $n \in \NN$, we have $d( \widetilde x, \varphi^n \widetilde x) \geq d( \widetilde x, \varphi^m \widetilde x)$. In other words, a point $\widetilde y \in \HH^3$ that $\varphi^n$ moves by distance $\epsilon$ would have to be \emph{closer} to the core geodesic than $\widetilde x$ is, hence
$\trad_{ m \lambda, \, m \tau \smod 2\pi}(\epsilon) \geq \trad_{ n \lambda, \, n \tau \smod 2\pi}(\epsilon)$.
(This includes the possibility that no such point $\widetilde y$ exists: that is, $\trad_{n\lambda,n\tau \smod 2\pi}(\epsilon) = - \infty$; see \refrem{WeirdArccosh}.) We conclude that
\[
r_{2\pi, \lambda, \tau}(\epsilon) = \trad_{ m \lambda, \, m \tau \smod 2\pi}(\epsilon)  = \max_{n \in \NN}  \big \{ \trad_{n\lambda, \, n\tau \smod 2\pi}(\epsilon) \big \}.
\]
For every $n \in \NN$, we may compute $\trad_{n\lambda, \, n\tau \smod 2\pi}(\epsilon) $ via Equation~\refeqn{TradDistance}:
\[
\trad_{n\lambda,\,  n\tau \!\smod 2\pi}(\epsilon) 
=    \arccosh \sqrt{ \frac{ \cosh \epsilon - \cos (n \tau)}{ \cosh (n\lambda) - \cos (n\tau)}    } ,
\]
where both sides might be $- \infty$ as in \refrem{WeirdArccosh}.

Now, assume that $N$ is a singular tube of cone angle $\alpha < 2\pi$. In this case, $N$ can be described as $N_{\alpha, \lambda, \tau} = \Hhat / \ZZ^2$, where $\ZZ^2  = \langle \varphi, \psi \rangle = \langle \varphi_{\lambda + i \tau}, \psi_{i\alpha} \rangle$. Let $x \in N$ be a point such that $2\injrad(x) = \epsilon$, and let $\widetilde x$ be a preimage in $\Hhat$. As above, \reflem{InjRadLoop} describes $\epsilon$ as a minimum of $d(\widetilde x, \eta \widetilde x)$ over all choices of $\eta \in \ZZ^2 \setminus \{ 0 \}$. We begin by restricting the values of $\eta$ that need to be considered.

Fix $n > 0$, and consider all isometries of the form $\eta = \varphi^n \psi^m$ as $m$ varies over $\ZZ$. All of these isometries translate the singular geodesic $\sigma \subset \Hhat$ by the same distance, so $d(\widetilde x, \eta \widetilde x)$ will be smallest when the rotational angle is smallest (in absolute value). Equivalently, the translation radius will be largest when the rotational angle is smallest. Thus
\[
\max_{m \in \ZZ} \big \{ \trad_{ n \lambda, n \tau+m\alpha}(\epsilon) \big \} = \trad_{ n \lambda, n \tau \smod \alpha}(\epsilon).
\]
Similarly, if $\eta = \varphi^0 \psi^m$ for $m \neq 0$, then $d(\widetilde x, \eta \widetilde x)$ will be smallest for $\eta = \psi$. We conclude that 
\[
  r_{\alpha,\lambda,\tau}(\epsilon) =\max\left\{ \trad_{0,\alpha}(\epsilon),\,
    \max_{n \in \NN}  \left\{ \trad_{n\lambda, \, n\tau \smod \alpha}(\epsilon) \right\}\right\}. 
\]

It remains to interpret the quantity $ \trad_{0,\alpha}(\epsilon)$ for various values of $\alpha$. If $0<\alpha<\pi$, a ball $B_{\epsilon/2}(\widetilde x)$ will be tangent to its image under $\psi$ without meeting the axis of $\Hhat$. In this case, $ \trad_{0,\alpha}(\epsilon)$
can be computed as in Equation~\refeqn{TradDistance}. If $\pi \leq \alpha < 2\pi$, the ball $B_{\epsilon/2}(\widetilde x)$ will bump directly into the core geodesic $\sigma \subset \Hhat$, hence $ \trad_{0,\alpha}(\epsilon) = \epsilon/2$.
\end{proof}

\section{Examples demonstrating sharpness}
\label{Sec:Sharpness}

In \refsec{TubeDepth} below, we will prove lower bounds on the radius of an $\epsilon$--tube in terms of $\epsilon$ and the core length $\lambda$. Quantitatively, the lower bound on $\cosh r(\epsilon)$ will be on the order of $\epsilon/\sqrt{\lambda}$. The following family of examples, suggested by Ian Biringer, shows that an $O(\epsilon/\sqrt{\lambda})$ bound is in fact optimal. As a consequence, we show in \refthm{Sharpness} that the lower bound of \refthm{EffectiveDistTubes} is sharp up to additive error.

\begin{proposition}
\label{Prop:Biringer}
For any $n \geq 4$, let $N_n = N_{2\pi, \lambda, \tau}$ be a nonsingular model solid torus whose core geodesic has
complex length
\[
 \lambda + i \tau = \frac{1}{n^2} + i \frac{2\pi }{n}.
\]
 Then, for every $\epsilon$ in the range $1.016/n \leq \epsilon \leq 0.3$, the tube radius in $N_n$ satisfies
\[
\cosh r(\epsilon) = \sqrt \frac{\cosh \epsilon - 1}{\cosh \sqrt{\lambda} - 1} 
\quad \in   \left( \frac{\epsilon}{\sqrt{\lambda} } , \:  1.004 \frac{\epsilon}{\sqrt{\lambda} } \right).
\]
\end{proposition}

Plugging the minimal and maximal values of $\epsilon$ into the above estimate shows that the radii appearing in \refprop{Biringer} range from $\arccosh(1.016) = 0.1767 \ldots $ to  $\arccosh (0.3 n)$, which goes to $\infty$  with $n$.
We suspect that similar behavior cannot hold for very small radii.

The proof of \refprop{Biringer} needs the following easy monotonicity result, which will be used repeatedly below.

\begin{lemma}
\label{Lem:EasyMonotonicity}
Let $a,b \in \RR$ be constants, and consider the function
\[
g(x) = \frac{a-x}{b-x} = \frac{x-a}{x-b}.
\]
Then $g(x)$ is strictly increasing in $x$ if $b < a$ and strictly decreasing in $x$ if $a < b$.
\end{lemma}

\begin{proof}
Compute the derivative. Alternately, consider $g$ as a linear fractional transformation of $\RR \PP^1 $, which preserves orientation if and only if $\det \left[ \begin{smallmatrix} 1 & -a \\ 1 & -b \end{smallmatrix} \right] > 0$.
\end{proof}

\begin{proof}[Proof of \refprop{Biringer}]
Fix $\epsilon \geq \sqrt{\lambda} = 1/n$. Let $r(\epsilon) = r_{2\pi, \lambda, \tau}(\epsilon) $  be the radius of the $\epsilon$--tube, as in \refdef{TubeRad}. We begin by proving a lower bound on $r(\epsilon)$,  using  \refprop{TubeRadFormula}:
\begin{align}
r_{2\pi, \lambda, \tau}(\epsilon) 
&  = \max_{m \in \NN}  \left\{ \trad_{m \lambda, \, m \tau \!\smod 2\pi}(\epsilon) \right\} \nonumber \\
& \geq \trad_{n \lambda, \, n \tau \!\smod 2\pi}(\epsilon) \nonumber \\
& = \trad_{\sqrt{\lambda}, \,  0}(\epsilon) \nonumber \\
& = \arccosh \sqrt{ \frac{ \cosh \epsilon - 1}{ \cosh \sqrt{\lambda} - 1} } \label{Eqn:BiringerLowerBound}  \, . 
\end{align}
We will eventually show that, for $\epsilon \geq 1.016/n$, the above inequality is actually equality. This amounts to showing that $n$ is indeed the power for $\epsilon$ (see \refdef{Power}). This requires a few estimates.

First, we get a tight two-sided estimate on the quantity in~\refeqn{BiringerLowerBound}. Assume that $1/n = \sqrt{\lambda} < \epsilon \leq 0.3$. The monotonicity of the function $h(x,y)$ in \reflem{CoshTaylorApprox} implies that
\[
1 \: < \: \frac{ \cosh \epsilon - 1} {\epsilon^2}  \cdot \frac{\lambda}{ \cosh \sqrt{\lambda} - 1} \: \leq \: \frac {\cosh (0.3) -1}{(0.3)^2} \cdot \frac{1}{1/2} \: = \:  (1.00375 \ldots)^2 .
\]
Multiplying all of this by $\epsilon^2/\lambda$ and taking square roots gives
\begin{equation}
\label{Eqn:TightEpsilonLambda}
\frac{\epsilon}{\sqrt{\lambda} } < \sqrt{ \frac{ \cosh \epsilon - 1}{ \cosh \sqrt{\lambda} - 1} } <  1.004 \frac{\epsilon}{\sqrt{\lambda} },
\end{equation}
as claimed in the statement of the Proposition.

Now, assume that $\epsilon \geq 1.016/n$.
Then~\refeqn{BiringerLowerBound} and~\refeqn{TightEpsilonLambda} combine to give
\[
\cosh r(\epsilon) \geq \sqrt{ \frac{ \cosh \epsilon - 1}{ \cosh \sqrt{\lambda} - 1} }
> \frac{\epsilon}{\sqrt{\lambda} }
= \epsilon n \geq 1.016,
\]
which implies
\begin{equation}
\label{Eqn:Tanh2Pi}
\tanh^2 r(\epsilon) > 0.0312.
\end{equation}

Next, we claim that (for any $\epsilon \geq \lambda$) the only possible powers  for $\epsilon$ are either $1$ or $n$. This can be seen from Equation~\refeqn{TradDistance}, substituting $\tau = 2\pi/n$:
\[
\cosh^2 \trad_{k \lambda, \, k \tau \!\smod 2 \pi}(\epsilon) =  \frac{ \cosh \epsilon - \cos(2\pi k /n)}{ \cosh (k \lambda) - \cos(2\pi k /n)} .
\]
When $k \in n \ZZ$, the subtracted term is $\cos(2\pi k/n) = 1$, hence constant, and the denominator is smallest for $k =n$. 
When $k \notin n \ZZ$, we have $\cos(2\pi k/n) \leq \cos(2\pi /n)$, hence \reflem{EasyMonotonicity} gives
\[
\cosh^2  \trad_{k \lambda, \, k \tau \!\smod 2 \pi}(\epsilon) \leq  \frac{ \cosh \epsilon - \cos(2\pi k/n)}{ \cosh ( \lambda) - \cos(2\pi k/n)} 
 \leq \frac{ \cosh \epsilon - \cos(2\pi /n)}{ \cosh ( \lambda) - \cos(2\pi /n)} .
\]
Thus only $k =1$ or $k =n$ can give a maximal value of $\trad$.

Finally, we claim that $n$ is indeed the power for  $\epsilon$. Let $\varphi = \varphi_{\lambda,\tau}$ be the loxodromic isometry of $\HH^3$ that generates the deck group of $N_n$. Fix $r = r(\epsilon)$, and let $\widetilde T_r$ be the lift to $\HH^3$ of the equidistant torus $T_r \subset N_n$. For $\widetilde x \in \widetilde T_r$, let $d_1 = d(\widetilde x, \varphi \widetilde x)$ and $d_n = d(\widetilde x, \varphi^n \widetilde x)$. By  \reflem{TubeTransDistance},
\begin{align*}
\cosh d_1 &= \cosh \lambda \cosh^2 r - \cos \tau \sinh^2 r \qquad = \cosh \left( \tfrac{1}{n^2} \right) \cosh^2 r - \cos \left( \tfrac{2\pi}{n} \right) \sinh^2 r \\
\cosh d_n &= \cosh (n \lambda) \cosh^2 r - \cos (n\tau) \sinh^2 r = \cosh \left( \tfrac{1}{n} \right) \cosh^2 r - \cos (0) \sinh^2 r
\end{align*}
hence
\[
\cosh d_n - \cosh d_1 = \left[  \cosh \left( \tfrac{1}{n} \right) - \cosh \left( \tfrac{1}{n^2} \right) \right] \cosh^2 r - \left[ 1 - \cos \left( \tfrac{2\pi}{n} \right)  \right] \sinh^2 r  .
\]
Using calculus, we check that when $n \geq 4$,
\[
\frac{    \cosh \left( \tfrac{1}{n} \right) - \cosh \left( \tfrac{1}{n^2} \right)  }{   1 - \cos \left( \tfrac{2\pi}{n} \right)  } \leq 0.02945 < \tanh^2 r,
\]
where the last inequality is by~\refeqn{Tanh2Pi}. Multiplying both sides by $\cosh^2 r$ gives 
\[
\left[   \cosh \left( \tfrac{1}{n} \right) - \cosh \left( \tfrac{1}{n^2} \right)  \right]
\cosh^2 r
 < \left[ 1 - \cos \left( \tfrac{2\pi}{n} \right)  \right] \sinh^2 r,
\]
hence $\cosh d_n < \cosh d_1$. This means that $\varphi^n$ translates points of $\widetilde T_r$ by \emph{less} than $\varphi$ for all radii satisfying~\refeqn{Tanh2Pi}, hence $n$ is the power for $\epsilon$.

We conclude that the inequality~\refeqn{BiringerLowerBound} is equality, as desired.
\end{proof}

We use \refprop{Biringer} to prove \refthm{Sharpness}, below. Recall from \refsec{TubeDistSharpness} that the upper bound of \refthm{Sharpness} differs from the lower bound of \refthm{EffectiveDistTubes} by an additive error of less than 2.2.

\begin{theorem}
\label{Thm:Sharpness}
Fix a pair $(\delta,\epsilon)$ such that 
 $0 < \sqrt{7.256 \delta} \leq \epsilon \leq 0.3$. Given the pair $(\delta,\epsilon)$, there is a model solid torus $N = N_{2\pi, \lambda, \tau}$ with $\lambda \leq \delta$, such that
\[
d_{2\pi, \lambda, \tau}(\delta, \epsilon) \leq \arccosh  \left( 1.116 \frac{\epsilon}{\sqrt{\delta} } \right).
\]
\end{theorem}

\begin{proof}
Given $\delta$ as above, let $n$ be the unique integer satisfying 
\[
\frac{1}{n} \leq \sqrt{\delta} < \frac{1}{n-1}.
\]
Now, set $\lambda + i \tau = 1/n^2 + 2\pi i  / n$, and let $N = N_n = N_{2\pi,\lambda,\tau}$ be a model solid torus as in \refprop{Biringer}. Since $\sqrt{\delta} \leq 0.3/\sqrt{7.256} < 0.1114$, we have $n \geq 9$. In addition, $\epsilon >  \sqrt{7 \delta} \geq \sqrt{7}/n$, verifying that the hypotheses of \refprop{Biringer} hold in our setting.

We will use \refprop{Biringer} to get an upper bound on $d_{2\pi, \lambda, \tau}(\delta, \epsilon)$. To do so, we need an upper bound on $\sqrt{\delta}/\sqrt{\lambda}$. Observe that 
\[
\frac{\sqrt{\delta}}{\sqrt{\lambda}} = n \sqrt{\delta} < \frac{n}{n-1}.
\]
Thus, if $n \geq 10$, we have $\sqrt{\delta/\lambda} \leq 10/9$. Meanwhile, if $n = 9$, recall that $\sqrt{\delta} < 0.1114$, which implies $\sqrt{\delta/\lambda} = 9 \sqrt{\delta} < 1.0026 < 10/9$ as well. Now, we compute:
\begin{align*}
d_{2\pi, \lambda, \tau}(\delta, \epsilon)
& = r_{2\pi, \lambda, \tau}( \epsilon) - r_{2\pi, \lambda, \tau}(\delta) \\
& \leq r_{2\pi, \lambda, \tau}( \epsilon) \\
& \leq \arccosh \left( 1.004 \, \frac{\epsilon}{\sqrt{\lambda}} \right) \\
& \leq \arccosh \left( 1.004 \cdot \frac{10}{9} \, \frac{\epsilon}{\sqrt{\delta}} \right). \qedhere
\end{align*}
\end{proof}

\section{Distance between tubes: upper bound}
\label{Sec:Upper}

In this section we prove the upper bound of \refthm{EffectiveDistTubes}.  This result requires some lemmas, the first of which is also used in the lower bound.

\begin{lemma}
\label{Lem:LowerBoundEasy}
Let $N = N_{\alpha, \lambda, \tau}$ be a model solid torus. Let $x,y$ be points of $N$ such that $2\, \injrad(x) = \delta > 0$ and $2 \, \injrad(y) = \epsilon > 0$. Then
\[
d(x,y) \geq \frac{\epsilon - \delta}{2}.
\]
\end{lemma}

\begin{proof}
Let $h = d(x,y)$. If $h \geq \epsilon/2$, there is nothing to prove. Thus we may assume that $h < \epsilon/2$. 
By \refdef{Injectivity}, there is an embedded ball $B = B_{\epsilon/2}(y)$ that is isometric to a ball in $\HH^3$. Since $h < \epsilon/2$, we have $x \in B$. By the triangle inequality, there is an embedded ball $B_{\epsilon/2 - h}(x)$ contained in $B$, implying
\[
 \injrad(x) = \delta/2 \geq \epsilon/2 - h. \qedhere
\]
\end{proof}

The following lemma controls the type of isometry that realizes injectivity radius in singular tubes. Recall that $T^\epsilon = \bdy N^{\leq \epsilon} \subset N$ denotes the equidistant torus consisting of points whose injectivity radius is exactly $\epsilon/2$.

\begin{lemma}
\label{Lem:EllipticWins}
Consider a singular solid torus $N = N_{\alpha, \lambda, \tau}$ of cone angle $\alpha < 2\pi$, and let $0 < \delta < \epsilon$. Suppose that, for $q \in T^\epsilon \subset N$, the injectivity radius $\injrad(q)$ is realized by an elliptic isometry $\psi_{i \alpha}$ (compare \reflem{InjRadLoop}). Then, for every $p \in T^\delta$, the injectivity radius $\injrad(p)$ is also realized by the same elliptic isometry $\psi_{i \alpha}$.
\end{lemma}

\begin{proof}
There are two cases: $\alpha < \pi$ and $\alpha \geq \pi$. We consider the latter case first.

Suppose that $\pi \leq \alpha < 2\pi$, and let $q \in T^\epsilon$. By \refprop{TubeRadFormula}, we must have $r(\epsilon) = \epsilon/2$, hence there is a point of intersection $z \in \Sigma \cap \overline{B_{\epsilon/2}(q)}$, where $\Sigma$ is the singular core of $N$. Let $\alpha$ be the geodesic segment of length $\epsilon/2$ from $q$ to $z$.

Let $p \in T^\delta$. Since the symmetry group of $N$ acts transitively on $T^\delta$, we may assume without loss of generality that $p \in \alpha \cap T^\delta $. By \reflem{LowerBoundEasy}, we have
\[
r(\delta) \: =  \: d(z,p) \: = \: d(z,q) - d(p,q) \: \leq \:  \frac{\epsilon}{2} - \frac{\epsilon - \delta}{2} \: = \: \frac{\delta}{2}.
\]
On the other hand, \refprop{TubeRadFormula} gives $r(\delta) \geq \delta/2$. 
Thus $r(\delta)  = \delta/2$, hence the injectivity radius of $T^\delta$ is realized by the elliptic generator $\psi_{i \alpha}$.

Next, suppose that $0<\alpha<\pi$. According to \refprop{TubeRadFormula}, the injectivity radius of $q \in T^\epsilon$ is realized by an elliptic isometry $\psi_{i \alpha}$ precisely when $\trad_{0,\alpha}(\epsilon) \leq \trad_{n \lambda, n \tau \smod \alpha}(\epsilon)$ for all $n \in \NN$.
Unwinding the definition of $\trad(\epsilon)$ via \reflem{TubeTransDistance}, we see that for  $n \geq 1$, 
\begin{equation}
\label{Eqn:EllipticWins}
\frac{\cosh\epsilon - \cos(n\tau\smod\alpha)}{\cosh (n \lambda) - \cos(n \tau\smod\alpha)} 
\leq 
\frac{\cosh \epsilon - \cos \alpha}{1 - \cos \alpha}.
\end{equation}
For simplicity of the following calculations, set
\[
x = \cosh\epsilon, \quad a = \cos\alpha, \quad b = 1, \quad a' = \cos(n\tau\smod\alpha), \quad b' = \cosh(n\lambda).
\]
We consider $a,b,a',b'$ to be constants in the following argument, because the integer $n$ will stay fixed. Note that $b < b'$ and $a < a'$. The inequality~\refeqn{EllipticWins} can be rewritten as
\[
k(x) \cdot \frac{x-a'}{b'-a'} = \frac{x - a}{b-a}, \mbox{ where } k(x) \geq 1.
\]
This, in turn, can be rewritten as
\[
k(x) \cdot \frac{b-a}{b'-a'} = \frac{x - a}{x-a'}.
\]
Since $a < a'$, \reflem{EasyMonotonicity} implies that $g(x) = \frac{x - a}{x - a'}$ is a strictly decreasing
function of $x$. Now, for any $0<\delta<\epsilon$, set $y = \cosh \delta$. We have
\[
k(y) \cdot \frac{b-a}{b'-a'} = \frac{y-a}{y-a'} > \frac{x-a}{x-a'} = k(x) \cdot \frac{b-a}{b'-a'}.
\]
We conclude that
\[
k(y) > k(x) \geq 1, \quad \mbox{that is} \quad \frac{y-a}{b-a} > \frac{y-a'}{b'-a'}.
\]
This means that for every $n \in \NN$, we have that $\trad_{n \lambda, n\tau\smod\alpha}(\delta)$ is strictly smaller than $\trad_{0,\alpha}(\delta)$. Thus $r(\delta) = \trad_{0,\alpha}(\delta)$, as desired.
\end{proof}

\begin{remark}
\label{Rem:PowerMonotonicity}
The above proof can be modified to establish the following ``monotonicity of powers'' statement. Suppose that $0 < \delta < \epsilon$, and the power for $\delta$ is $n \in \NN$. That is, for $p \in T^\delta$, the injectivity radius $\injrad(p)$ is realized by a loxodromic isometry $\varphi_{\lambda+ i \tau}^n$. Then the power for $\epsilon$ is at least $n$. Note that by \reflem{EllipticWins}, the realizing isometry must be loxodromic. The proof starts with  the following analogue of~\refeqn{EllipticWins}: for all $1 \leq m < n$, we have
\begin{equation}
\label{Eqn:PowerMonotonic} 
\frac{\cosh \delta - \cos (m\tau \smod \alpha)}{\cosh (m\lambda) - \cos (m\tau \smod \alpha) } \leq \frac{\cosh \delta - \cos (n \tau \smod \alpha)}{\cosh (n \lambda) - \cos (n \tau \smod \alpha) }.
\end{equation}
Now,~\refeqn{PowerMonotonic} combined with the monotonicity of the function $g(x)$ of \reflem{EasyMonotonicity} (in the opposite direction compared to the last proof) will imply that~\refeqn{PowerMonotonic} also holds with  $\epsilon$ instead of $\delta$. Since we do not need this statement, we omit the details.
\end{remark}

We can now prove the following statement, which will quickly imply the upper bound of \refthm{EffectiveDistTubes}.

\begin{proposition}
\label{Prop:MultiplicativeRadGap}
Suppose $0 < \delta < \epsilon$ and $0<\alpha \leq 2\pi$. Then, in any model solid torus $N = N_{\alpha,\lambda,\tau}$ such that $N^{\leq \delta} \neq \emptyset$, we have
\[
\frac{\cosh r_{\alpha, \lambda, \tau}(\epsilon)}{\cosh r_{\alpha, \lambda, \tau}(\delta)} \leq  \sqrt{ \frac{\cosh \epsilon - 1}{\cosh \delta - 1}  }.
\]
Equality holds if and only if the injectivity radii of $T^\delta$ and $T^\epsilon$ are realized by the same loxodromic isometry, whose rotational part is trivial. In particular, if $\alpha = 2\pi$ and $\tau = 0$, then
 equality holds.
\end{proposition}

\begin{proof}
We consider three cases, depending on the value of $\alpha$ and the power for $\epsilon$.

First, suppose that the power for $\epsilon$ is $m \in \NN$, hence the injectivity radius of $T^\epsilon$ is realized by a loxodromic isometry $\varphi_{\lambda + i \tau}^m$. In the following computation, using \refprop{TubeRadFormula}, the ellipsis $(\cdots)$ denotes any elliptic terms that arise when $\alpha < 2\pi$.
\begin{align*}
\frac{\cosh r_{\alpha, \lambda, \tau}(\epsilon)}{\cosh r_{\alpha, \lambda, \tau}(\delta)}
& = \frac{ \max \left\{  \max_{n \in \NN}  \left\{ \cosh \trad_{n \lambda, n \tau \smod \alpha}(\epsilon) \right\} , \cdots \right\}  }{ \max \left\{ \max_{n \in \NN}  \left\{ \cosh \trad_{n \lambda, n \tau \smod \alpha}(\delta) \right\} , \cdots \right\}  }  \quad \text{by \refprop{TubeRadFormula}} \\
& = \frac{   \cosh \trad_{m \lambda, m \tau \smod \alpha}(\epsilon)   }{ \max \left\{ \max_{n \in \NN}  \left\{ \cosh \trad_{n \lambda, n \tau \smod \alpha}(\delta) \right\} , \cdots \right\}  } \quad \text{by definition of $m$} \\
& \leq  
 \frac{  \cosh \trad_{m \lambda, m \tau \smod \alpha }(\epsilon)  }{\cosh \trad_{m \lambda,  m \tau \smod \alpha}(\delta)   } 
 \\
& =    \sqrt{ \frac{ \cosh \epsilon - \cos (m\tau \smod \alpha)}{ \cosh m\lambda - \cos (m\tau \smod \alpha)} \cdot \frac{ \cosh m\lambda - \cos (m \tau \smod \alpha)}{ \cosh \delta - \cos (m\tau \smod \alpha)}     } \quad \text{ by~\refeqn{TradDistance}}  \\
& =   \sqrt{ \frac{ \cosh \epsilon - \cos (m\tau \smod \alpha)} { \cosh \delta - \cos (m\tau \smod\alpha)}     }  \\
& \leq  \sqrt{ \frac{ \cosh \epsilon - 1} { \cosh \delta - 1} }  \quad \text{by \reflem{EasyMonotonicity}}.
\end{align*}
Observe that the first inequality is equality precisely when $m$ is also the power for $\delta$. The second inequality is equality precisely when $(m \tau \smod \alpha) = 0$: that is, when the realizing isometry has trivial rotational part.

Next, suppose that the injectivity radius of $T^\epsilon$ is realized by an elliptic isometry $\psi_{ i \alpha}$, and furthermore $0 < \alpha < \pi$. Then \reflem{EllipticWins} says that the injectivity radius of $T^\delta$ is realized by the same elliptic. Thus we may compute as above:
\begin{align*}
\frac{\cosh r_{\alpha, \lambda, \tau}(\epsilon)}{\cosh r_{\alpha, \lambda, \tau}(\delta)}
& =  
 \frac{  \cosh \trad_{0, \alpha}(\epsilon)  }{\cosh \trad_{0,  \alpha}(\delta)   } 
 \\
& =    \sqrt{ \frac{ \cosh \epsilon - \cos ( \alpha)}{ 1 - \cos ( \alpha)} \cdot \frac{ 1 - \cos ( \alpha)}{ \cosh \delta - \cos ( \alpha)}     } \quad \text{ by~\refeqn{TradDistance}}  \\
& =   \sqrt{ \frac{ \cosh \epsilon - \cos ( \alpha)} { \cosh \delta - \cos (\alpha)}     }  \\
& <  \sqrt{ \frac{ \cosh \epsilon - 1} { \cosh \delta - 1} } .
\end{align*}
The inequality is strict because $0 <  \alpha < \pi$, hence $\cos \alpha < 1$.

Finally, suppose that the injectivity radius of $T^\epsilon$ is realized by an elliptic isometry $\psi_{ i \alpha}$, and furthermore $\pi \leq \alpha < 2\pi$. Then \reflem{EllipticWins} says that the injectivity radius of $T^\delta$ is realized by the same elliptic. Furthermore, $r(\epsilon) = \epsilon/2$ and $r(\delta) = \delta/2$. Thus
\begin{align*}
\frac{\cosh^2 (r_{\alpha, \lambda, \tau}(\epsilon))}{\cosh^2 (r_{\alpha, \lambda, \tau}(\delta))}
& = \frac{\cosh^2 (\epsilon/2)}{\cosh^2 (\delta/2)} \\
& < \frac{\cosh^2 (\epsilon/2) -1}{\cosh^2 (\delta/2) -1 }, \quad \text{by \reflem{EasyMonotonicity} } \\
& = \frac{(2 \cosh^2 (\epsilon/2) -1) - 1 }{(2 \cosh^2 (\delta/2) -1) -1} \\
& = \frac{\cosh (\epsilon) - 1 }{\cosh (\delta)  -1}.
\end{align*}
Again, the inequality is strict in this case.
\end{proof}

We can now prove the upper bound of \refthm{EffectiveDistTubes}, including its sharpness. 

\begin{proposition}
\label{Prop:EffectiveDistUpper}
Suppose $0 < \delta < \epsilon$ and $0 < \alpha \leq 2\pi$. Then, in any model solid torus $N_{\alpha,\lambda,\tau}$ with $N^{\leq \delta}$ non-empty, we have
\begin{equation*}
d_{\alpha,\lambda, \tau}(\delta,\epsilon) \leq \arccosh \sqrt{ \frac{\cosh \epsilon - 1}{\cosh \delta - 1}  }.
\end{equation*}
Furthermore, 
equality holds if and only if  $(\alpha,\lambda,\tau) = (2\pi, \delta, 0)$.
\end{proposition}

\begin{proof}
Combining \reflem{CoshMult} and \refprop{MultiplicativeRadGap} gives
\begin{equation}
\label{Eqn:DistUpperCompute}
\cosh d_{\alpha,\lambda, \tau}(\delta,\epsilon) = \cosh \big( r_{\alpha, \lambda, \tau}(\epsilon) -  r_{\alpha, \lambda, \tau}(\delta)  \big) \leq  \frac{\cosh r_{\alpha, \lambda, \tau}(\epsilon)}{\cosh r_{\alpha, \lambda, \tau}(\delta)} \leq  \sqrt{ \frac{\cosh \epsilon - 1}{\cosh \delta - 1}  }.
\end{equation}

Next, let us analyze when equality can hold. Assuming $\delta < \epsilon$, hence $r(\delta) < r(\epsilon)$, \reflem{CoshMult} implies that the first inequality of~\refeqn{DistUpperCompute} will be equality if and only if $r(\delta) = 0$. But $r(\delta) \geq \delta/2 > 0$ when $N$ is singular, hence equality cannot hold for singular tori.

From now on, assume that $\alpha = 2\pi$, hence $N$ is nonsingular. By \reflem{CoshMult},  the first inequality of~\refeqn{DistUpperCompute} will be equality precisely when $r(\delta) = 0$, hence when $\lambda = \delta$ and the realizing isometry is a generator of $\pi_1(N) = \ZZ$. By \refprop{MultiplicativeRadGap}, the second inequality of~\refeqn{DistUpperCompute} is equality precisely when this realizing isometry has rotational part $\tau = 0$.
\end{proof}

\section{Euclidean bounds}
\label{Sec:Euclidean}

We now begin working toward the lower bound of \refthm{EffectiveDistTubes}. In that argument, we will need to estimate distances on the torus $T^\epsilon$ that forms the boundary of an $\epsilon$--tube.
The torus $T^\epsilon \subset N$ inherits a Euclidean metric, as does its preimage $\widetilde{T}^\epsilon \subset \Hhat$, and we may use that metric to measure the distance between points. In this section, we consider how the Euclidean distance between points on $\widetilde{T}^\epsilon$ relates to the actual hyperbolic distance between these points in $\Hhat$. 

For points $p$ and $q$ in an equidistant torus $\widetilde{T}^\epsilon = \widetilde{T}_r \subset \Hhat$, define $d_E(p,q)$ to be the distance between them in the Euclidean metric on $\widetilde{T}_r$. Note that if $p=(r,0,0)$ and $q=(r,\zeta,\theta)$ in cylindrical coordinates, then~\refeqn{CylindricalCoords} gives
\begin{equation}
\label{Eqn:EuclDist}
  d_E(p,q)^2 = \zeta^2\cosh^2 r + \theta^2\sinh^2 r.
\end{equation}

\begin{lemma}
\label{Lem:EucInjectivityGeneral}
Let $\widetilde{T}_r \subset \Hhat$ be a plane at fixed distance $r > 0$ from the singular geodesic $\hat{\sigma}$.
Let $p,q \in \widetilde{T}_r$ be points whose $\theta$--coordinates differ by at most $A \leq \pi$ and whose $\zeta$--coordinates differ by at most $B$. Then
\[
\frac{1 - \cos A}{A^2} \, d_E (p,q)^2 \leq \cosh d(p,q) - 1 \leq \frac{\cosh B - 1}{B^2} \, d_E (p,q)^2.
\]
\end{lemma}

\begin{proof}
Suppose without loss of generality that the cylindrical coordinates of $p$ are $(r,0,0)$ and the coordinates of of $q$ are $(r,\zeta, \theta)$, where  $|\theta | \leq A$. We can now compute using \reflem{TubeTransDistance}:
\begin{align}
\cosh d(p,q) - 1
& =  \cosh \zeta \cosh^2 r  - \cos \theta  \sinh^2 r - 1 \nonumber \\
& =  (\cosh \zeta - 1) \cosh^2 r + (\cosh^2 r -1) - \cos \theta  \sinh^2 r \nonumber \\
& =  (\cosh \zeta - 1) \cosh^2 r + (1 - \cos \theta )  \sinh^2 r  \nonumber \\
& = \frac{\cosh \zeta - 1}{\zeta^2} (\zeta \cosh r)^2 + \frac{1 - \cos \theta}{\theta^2} (\theta \sinh r)^2.
\label{Eqn:DivideByZero}
\end{align}
If either $\zeta = 0$ or $\theta = 0$, equation~\refeqn{DivideByZero} should be interpreted by substituting the limiting value of $1/2$:
\begin{equation}
\label{Eqn:QuadraticLimit}
\lim_{\zeta \to 0}  \frac{\cosh \zeta - 1}{\zeta^2} = \frac{1}{2} = \lim_{\theta \to 0} \frac{1 - \cos \theta}{\theta^2}.
\end{equation}
The monotonicity of the functions in~\refeqn{QuadraticLimit} leads to the following chain of inequalities:
\[
\frac{1 - \cos A}{A^2} \leq \frac{1 - \cos \theta}{\theta^2}   \leq \frac{1}{2} \leq  \frac{\cosh \zeta - 1}{\zeta^2} \leq  \frac{\cosh B - 1}{B^2}. 
\]
To get the lower bound of the lemma, we return to~\refeqn{DivideByZero}, and substitute the lower bound $(1 - \cos A)/(A^2)$ for both $(1-\cos\theta)/(\theta^2)$ and $(\cosh\zeta-1)/(\zeta^2)$. The upper bound follows similarly.
\end{proof}

The last result should be interpreted in light of  \reflem{CoshTaylorApprox}: so long as $d(p,q)$ is bounded above, $(\cosh d(p,q)-1)$ is not too different from $\frac{1}{2} d(p,q)^2$. Thus  \reflem{EucInjectivityGeneral} can be interpreted as saying that the Euclidean and hyperbolic distances are quite similar. The next lemma makes this idea precise, in a special case.

\begin{lemma}
\label{Lem:EucInjectivity}
Let $\widetilde{T}_r \subset \Hhat$ be a plane at fixed distance $r > 0$ from the singular geodesic $\hat{\sigma}$.
Let $p,q \in \widetilde{T}_r$ be points whose $\theta$--coordinates differ by at most $A \leq \pi$. 
Suppose that $d(p,q) \leq 0.3$. Then
\[ 
0.6342 \, d_E(p,q) < d(p,q) < d_E(p,q).
\]
\end{lemma}

\begin{proof}
The upper bound $d(p,q) < d_E(p,q)$ is immediate, because tubes are strictly convex. 
For the lower bound, we use the monotonicity of the function
 $f(x) = (\cosh x - 1)/(x^2)$.
\reflem{CoshTaylorApprox}
 implies that on the interval $[0, \, 0.3]$, we have
\begin{equation}
\label{Eqn:CoshTaylorApprox}
\frac{\cosh d(p,q) - 1}{d(p,q)^2} \leq f(0.3) =  0.50376 \ldots .
\end{equation}
Similarly, $g(x) = (1 - \cos x)/(x^2)$ is strictly decreasing on $[0,\pi]$, hence (setting $A$ as in \reflem{EucInjectivityGeneral}) we get
\begin{equation}
\label{Eqn:CosTaylorApprox}
\frac{1 - \cos A}{A^2} \geq g(\pi) =  \frac{2}{\pi^2}.
\end{equation}
Plugging~\refeqn{CoshTaylorApprox} and~\refeqn{CosTaylorApprox} into  \reflem{EucInjectivityGeneral} gives
\[
0.50377 \, d(p,q)^2 \geq \cosh d(p,q) - 1 \geq \frac{1 - \cos A}{A^2} d_E (p,q)^2 \geq   \frac{2}{\pi^2} d_E (p,q)^2 ,
\]
which simplifies to the desired result.
\end{proof}

\section{The depth of a tube}
\label{Sec:TubeDepth}

The next step in the proof of \refthm{EffectiveDistTubes} is \refprop{TubeRadBound}, which provides a lower bound on the radius of an $\epsilon$--tube.  Note that by \refprop{Biringer}, the estimate of \refprop{TubeRadBound} is sharp up to multiplicative error in $\cosh r(\epsilon)$, hence up to additive error in $r(\epsilon)$.

\begin{proposition}
\label{Prop:TubeRadBound}
Let $N = N_{\alpha, \lambda, \tau}$ be a model solid torus whose core curve has cone angle $0 < \alpha \leq 2\pi$ and length $0 < \lambda \leq 2.97$.  Suppose that $N^{\leq \epsilon} \neq \emptyset$. Then the tube radius $r(\epsilon)$ satisfies
\[
\cosh r_{\alpha, \lambda, \tau}(\epsilon) \: \geq \:  \frac{\epsilon}{\sqrt{(4 \pi  / \sqrt{3}) \, \lambda }}
\: \geq \:  \frac{\epsilon}{\sqrt{7.256 \, \lambda}} \,  .
\]
\end{proposition}

The proof breaks into separate cases, depending on whether $N$ is singular or nonsingular. We will handle singular tubes first.

\begin{lemma}
\label{Lem:SingularTubeRad}
Let $N = N_{\alpha, \lambda, \tau}$ be a model solid torus whose core has cone angle $\alpha < 2\pi$. Then 
\[
\sinh 2 r_{\alpha,\lambda,\tau}(\epsilon) 
\: \geq \:  \frac{\sqrt{3} \, \epsilon^2}{ \alpha \lambda}
\: > \: \frac{\sqrt{3} \, \epsilon^2}{ 2 \pi \lambda}.
\]
\end{lemma}

\begin{proof}
Let $T_r \subset N$ be the torus at radius $r$ from the core, and let $\widetilde T_r$ be is preimage in $\Hhat$.
By Equation~\refeqn{CylindricalCoords}, the area of $T_r$ is 
\begin{equation}
\label{Eqn:TorusArea}
\area (T_r) = \alpha \lambda \sinh r \cosh r = \frac{\alpha \lambda}{2} \sinh 2r.
\end{equation}

Now, suppose that an (arbitrary) point $x \in T_r$ has $\injrad(x) = \epsilon/2$. Let $\widetilde x \in \widetilde T_r$ be a lift of $x$. By \reflem{InjRadLoop}, every non-trivial deck transformation $\eta \in \ZZ^2 - \{0\}$ satisfies
\[
d(\widetilde x, \eta \widetilde x) \geq \epsilon.
\]
This includes elliptic isometries of $\Hhat$, because $N$ is singular.
Observe that Euclidean distance along $\widetilde T_r$ is greater than hyperbolic distance, because tubes are strictly convex (compare \reflem{EucInjectivity}). Thus 
\[
d_E (\widetilde x, \eta \widetilde x) > d(\widetilde x, \eta \widetilde x) \geq \epsilon.
\]
This means $\widetilde x$ is the center of a Euclidean disk of radius $\epsilon/2$, disjoint from all of its translates by the deck group. Projecting down to $T_r$ gives an embedded Euclidean disk of radius $\epsilon/2$.
Since a packing of $\RR^2$ by isometric disks has density at most $\pi/(2\sqrt{3})$, we have
\begin{equation}
\label{Eqn:EpsilonArea}
\area(T^\epsilon) \: \geq \: 
\frac{\pi \epsilon^2}{4}  \cdot \frac{2 \sqrt{3}}{\pi}
\: = \:
\frac{\sqrt{3}}{2}  \epsilon^2.
\end{equation}

Combining this result with Equation~\refeqn{TorusArea}, we obtain
\[
\area(T^\epsilon) = \frac{\alpha \lambda}{2} \sinh 2 r(\epsilon) \geq \frac{\sqrt{3} \, \epsilon^2}{2}.
\qedhere
\]
\end{proof}

The nonsingular case of \refprop{TubeRadBound} uses the following result by Cao, Gehring, and Martin \cite[Lemma 3.4]{CaoGehringMartin}, sharpening an earlier lemma by Zagier \cite[Page 1045]{meyerhoff}. We remark that more elementary forms of \reflem{ZagierCGM}, such as a pigeonhole principle argument due to Meyerhoff \cite[Page 1048]{meyerhoff}, suffice to show a version of \refprop{TubeRadBound}, but with a larger constant in front of $\sqrt{\lambda}$. 

\begin{lemma}
\label{Lem:ZagierCGM}
Consider a nonsingular tube whose core geodesic has  complex length $\lambda + i\tau$, where $0 < \lambda \leq 2.97$. Then there is an integer $m \geq 1$ such that
\begin{equation*}
\cosh(m \lambda) - \cos (m \tau) 
   \leq 
\frac{  2 \pi \lambda }{ \sqrt{3} } \leq 3.628 \lambda.
\end{equation*}
   \end{lemma}

\begin{proof}
This is a restatement of \cite[Lemma 3.4]{CaoGehringMartin}.  To convert their result to the form stated above, one needs the identity
\[
\left|  2 \sinh^2 \left( \frac{m(\lambda + i \tau)}{2} \right) \right|
= \left| \cosh \big( m(\lambda + i \tau) \big) - 1 \right| = \cosh(m \lambda) - \cos (m \tau) ,
\]
which is readily verified. See \cite[Equation (3.10)]{CaoGehringMartin}.
\end{proof}

Using this, we can prove the nonsingular case of \refprop{TubeRadBound}.

\begin{lemma}
\label{Lem:NonSingTubeRad}
Let $r_{2\pi,\lambda, \tau}(\epsilon)$ be the radius of a nonsingular tube that has injectivity radius $\epsilon$. Suppose the core curve has length $\lambda \leq 2.97$.  Then, for every $\epsilon \geq \lambda$,
\[
\cosh^2 ( r_{2\pi,\lambda, \tau}(\epsilon) )\geq  \frac{\sqrt{3} \, \epsilon^2}{4 \pi  \lambda}.
\]
\end{lemma}

\begin{proof}
Let $m \geq 1$ be the integer guaranteed by \reflem{ZagierCGM}. Then \refprop{TubeRadFormula} gives
\begin{align*}
\cosh^2(r_{2\pi,\lambda, \tau}(\epsilon))
&=  \max_{n \in \NN}  \, \{ \cosh^2(\trad_{n\lambda, \, n \tau \smod \alpha}(\epsilon)) \} \\
& \geq \cosh^2(\trad_{m \lambda, \, m \tau \smod \alpha}(\epsilon))  \\
  & = \frac{\cosh\epsilon - \cos(m\tau )}{\cosh(m\lambda) - \cos(m\tau )} \\
  & \geq (\cosh\epsilon - 1) \cdot \frac{\sqrt{3}}{2\pi \lambda}  \\
  & \geq \frac{\epsilon^2 }{2} \cdot \frac{\sqrt{3}}{2\pi \lambda} .  \qedhere
\end{align*}
\end{proof}

\begin{proof}[Proof of \refprop{TubeRadBound}]
If $N$ is a nonsingular tube whose core has length $\lambda \leq 2.97$, the proposition holds by  \reflem{NonSingTubeRad}. Meanwhile, if $N$ is a singular tube whose core curve has cone angle $\alpha < 2\pi$, 
\reflem{SingularTubeRad} implies
\[
\cosh^2 r(\epsilon)
> \sinh r(\epsilon) \cosh r(\epsilon) 
 = \frac{1}{2} \sinh 2 r(\epsilon) 
 \geq \frac{\sqrt{3} \, \epsilon^2}{4 \pi \lambda} . \qedhere
\]
\end{proof}

\smallskip

\section{Distance between tubes: lower bound}
\label{Sec:Lower}

In this section, we prove the lower bound of \refthm{EffectiveDistTubes}. The proof breaks into two cases: shallow and deep.
A tube is said to be \emph{shallow} if its radius is bounded above by some constant denoted $\rmax$. Similarly, a tube is said to be \emph{deep} if its radius is bounded below by some constant denoted $\rmin$. The optimal values of $\rmin$ and $\rmax$ will be determined later.

The following lemma gives a bound for shallow tubes.

\begin{lemma}
\label{Lem:LowerBoundShallow}
    Suppose that $0 < \delta < \epsilon$.  Let $N = N_{\alpha, \lambda, \tau}$ be a model solid torus, where $0 < \alpha \leq 2\pi$ and $0 < \lambda \leq \min(\delta, \, 2.97)$. Suppose as well that
 $r_{\alpha, \lambda, \tau}(\delta) \leq \rmax$, for some $\rmax>0$. Then
\[
d_{\alpha, \lambda, \tau}(\delta,\epsilon) \geq 
\arccosh \frac{\epsilon}{\sqrt{7.256 \, \delta}} - \rmax.
\]
\end{lemma}

\begin{proof}
An immediate consequence of \refprop{TubeRadFormula} is that the tube radius $r_{\alpha, \lambda, \tau}(\epsilon)$ is decreasing in $\lambda$.
Combining this fact with \refprop{TubeRadBound} gives
\begin{align*}
d_{\alpha, \lambda, \tau}(\delta,\epsilon) 
&= r_{\alpha, \lambda, \tau}(\epsilon) - r_{\alpha, \lambda, \tau}(\delta) \\
&\geq r_{\alpha, \delta, \tau}(\epsilon) - r_{\alpha, \lambda, \tau}(\delta) \\
&\geq \arccosh \frac{\epsilon}{\sqrt{7.256 \, \delta}} - \rmax. \qedhere
\end{align*}
\end{proof}

The corresponding statement for deep tubes is:

\begin{lemma}
\label{Lem:LowerBoundDeep}
Suppose that $0 < \delta < \epsilon$, where $\delta \leq 0.3$. 
Let $0 < \alpha \leq 2\pi$, and let $N = N_{\alpha, \lambda, \tau}$ be a model solid torus such that  $N^{\leq \delta} \neq \emptyset$.
Suppose as well that $r_{\alpha, \lambda, \tau}(\delta) \geq \rmin > 0$.
Then 
\[
d_{\alpha, \lambda, \tau}(\delta,\epsilon) 
\: \geq \: \log \left( \frac{ \epsilon }{ \delta} \cdot 1.268 \sinh \rmin \right)  - \rmin .
\]
\end{lemma}

To prove \reflem{LowerBoundDeep}, we need to compute how fast the Euclidean injectivity radius changes.

\begin{lemma}
\label{Lem:TorusProjection}
Let $0 < r < R$, and consider equidistant tori $T_r, T_R \subset N_{\alpha, \lambda, \tau}$. Let $c_r \subset T_r$ be a rectifiable curve, and let $c_R \subset T_R$ be the cylindrical projection of $c_r$ to $T_R$. Then
\[
\frac{\cosh R}{\cosh r} \leq \frac{\ell(c_R)}{\ell(c_r)} \leq \frac{\sinh R}{\sinh r}.
\]
\end{lemma}

\begin{proof}
Let $(r,\zeta(t),\theta(t))$ be a parametrization of $c_r$, where $t \in [0,1]$. By~\refeqn{CylindricalCoords}, the distance element on $T_r$ satisfies
\begin{equation*}
ds^2 = \cosh^2 r \, d\zeta^2 + \sinh^2 r \, d\theta^2 .
\end{equation*}
Thus we may compute that
\[
\ell(c_r) = \int_0^1 \, ds = \int_0^1 \sqrt{\cosh^2 r \left( \frac{d\zeta}{dt} \right)^2 + \sinh^2 r \left( \frac{d\theta}{dt} \right)^2 } dt .
\]
Similarly, 
\begin{align*}
\ell(c_R)  &= \int_0^1 \sqrt{\cosh^2 R \left( \frac{d\zeta}{dt} \right)^2 + \sinh^2 R \left( \frac{d\theta}{dt} \right)^2 } dt \\
& =  \int_0^1  \sqrt{\frac{\cosh^2 R}{\cosh^2 r} \cdot \cosh^2 r \left( \frac{d\zeta}{dt} \right)^2 + \frac{\sinh^2 R}{\sinh^2 r} \cdot \sinh^2 r \left( \frac{d\theta}{dt} \right)^2 } dt \\
& \geq \int_0^1 \frac{\cosh R}{\cosh r} \, \sqrt{\cosh^2 r \left( \frac{d\zeta}{dt} \right)^2 + \sinh^2 r \left( \frac{d\theta}{dt} \right)^2 } dt \\
& = \frac{\cosh R}{\cosh r} \, \ell(c_r),
\end{align*}
where the inequality in the next-to-last line is \reflem{SinhCoshGrowth}.

A similar computation proves the other inequality in the lemma.
\end{proof}

\begin{proof}[Proof of \reflem{LowerBoundDeep}]
Let $x \in T^\delta \subset N$. Then, by \reflem{InjRadLoop}, 
there is a lift $\widetilde{x} \in \Hhat$ and a deck transformation $\eta$ such that $d(\widetilde{x}, \eta \widetilde{x}) = \delta$. 
Furthermore, $\eta$ is a loxodromic if $N$ is nonsingular.
The points $\widetilde{x}, \eta (\widetilde{x})$ are connected by a Euclidean geodesic arc $\widetilde{c}^\delta \subset \widetilde{T}^\delta$.
Projecting back down to $N$, we have a Euclidean geodesic $c^\delta$ whose length satisfies
\begin{equation}
\label{Eqn:DeltaLoopLength}
0.634 \, \ell(c^\delta) = 0.634 \, d_E(\widetilde{x}, \eta \widetilde{x}) < \delta ,
\end{equation}
where the inequality is by \reflem{EucInjectivity}. 

Let $c^\epsilon \subset T^\epsilon$ be the cylindrical projection of $c^\delta$ to $T^\epsilon$. Since every point of $T^\epsilon$ has injectivity radius $\epsilon/2$, we know that $\ell(c^\epsilon) \geq \epsilon$. 
Therefore,
\[
\frac{\epsilon}{\delta}
 < \: \frac{\ell(c^\epsilon)}{ 0.634 \, \ell(c^\delta) } 
 \leq \: \frac{\sinh r(\epsilon)}{0.634 \, \sinh r(\delta)} 
 < \: \frac{e^{r(\epsilon)}}{2 \cdot 0.634 }
\cdot  \frac{e^{-r(\delta)} \cdot e^{r(\delta)}  }{ \sinh r(\delta)} 
 \leq \: \frac{e^{r(\epsilon)-r(\delta)}}{1.268} \cdot \frac{ e^{\rmin}  }{ \sinh \rmin} .
\]
Here, the first inequality uses~\refeqn{DeltaLoopLength}, the second inequality is \reflem{TorusProjection}, the third inequality uses the definition of $\sinh r$, and the fourth inequality uses the monotonicity of $e^{-r} \sinh r$.
Taking logarithms completes the proof.
\end{proof}

To complete the proof of \refthm{EffectiveDistTubes}, we need one more elementary lemma.

\begin{lemma}
\label{Lem:ComparisonFunction}
Consider the function
\[
j(\delta, \epsilon) = \frac{1}{1.268} \left( \sqrt{ \frac{\delta}{7.256} } + \sqrt{ \frac{\delta}{7.256} - \frac{\delta^2}{\epsilon^2} } \right) .
\]
On the domain $Q = \{  (\delta, \epsilon) : 0 \leq \epsilon \leq 0.3, \, 0 \leq \delta \leq \epsilon^2 / 7.256  \}$, the function satisfies
$j(\delta,\epsilon) \leq 0.0424$.
\end{lemma}

\begin{proof}
It is clear from the definition that $j(\delta, \epsilon)$ is increasing in $\epsilon$. Thus the maximum over $Q$ occurs on the arc of $\bdy Q$ where $\epsilon = 0.3$. On this arc, we compute $\frac{\partial}{\partial \delta} j(\delta, \epsilon)$ and find that $j(\delta, \epsilon)$ has a single interior critical point at $\delta = 0.0093026 \ldots$, with maximal value $j(\delta, \epsilon) = 0.042357 \ldots$.
\end{proof}

We can now complete the proof of the main theorem.

\begin{proof}[Proof of \refthm{EffectiveDistTubes}]
The upper bound of the theorem is proved in \refprop{EffectiveDistUpper}.

For the lower bound, suppose that $0 < \delta < \epsilon \leq 0.3$, and set $\rmax = \rmin = 0.0424$. 
Suppose that $N = N_{\alpha,\lambda,\tau}$ is a model solid torus such that $\lambda \leq \delta$.
By \reflem{LowerBoundEasy}, we have $ d_{\alpha,\lambda, \tau}(\delta,\epsilon) \geq (\epsilon - \delta)/ 2$. Thus it remains to show that
\begin{equation}
\label{Eqn:DesiredBoundTubes}
d_{\alpha,\lambda, \tau}(\delta,\epsilon)
\geq
\arccosh \frac{\epsilon}{\sqrt{7.256 \, \delta}} - \rmin.
\end{equation}
If $ \sqrt{7.256 \, \delta} > \epsilon$, our convention (see \refrem{WeirdArccosh}) is that $\arccosh(\epsilon/\sqrt{7.256 \, \delta}) = - \infty$, 
hence~\refeqn{DesiredBoundTubes} holds trivially.
If $r(\delta) \leq \rmin$, the desired lower bound of~\refeqn{DesiredBoundTubes} holds by \reflem{LowerBoundShallow}. 
From now on, we assume that  $r(\delta) \geq \rmin = 0.0424$ and $ \delta \leq \epsilon^2 / 7.256 \leq 0.3^2 / 7.256$. This means the hypotheses of both \reflem{LowerBoundDeep} and  \reflem{ComparisonFunction} are satisfied.

We will use  \reflem{ComparisonFunction} to show that the lower bound of \reflem{LowerBoundDeep} is stronger than~\refeqn{DesiredBoundTubes}. We compute as follows, starting from  \reflem{ComparisonFunction}:

\begin{align*}
0.0424 = \rmin
&\geq \frac{1}{1.268} \left( \sqrt{ \frac{\delta}{7.256} } + \sqrt{ \frac{\delta}{7.256} - \frac{\delta^2}{\epsilon^2} } \right)  \\
 \sinh \rmin
& \geq  \frac{1}{1.268} \cdot \frac{\delta}{\epsilon} \left( \sqrt{ \frac{\epsilon^2}{7.256 \delta} } + \sqrt{ \frac{\epsilon^2}{7.256 \delta} - 1 } \right) \\
 \frac{\epsilon }{\delta} \cdot 1.268 \sinh \rmin
& \geq \sqrt{ \frac{\epsilon^2}{7.256 \delta} } + \sqrt{ \frac{\epsilon^2}{7.256 \delta} - 1 } \\
 \log \left(  \frac{\epsilon }{\delta} \cdot 1.268  \sinh \rmin  \right) - \rmin
& \geq \log \left( \sqrt{ \frac{\epsilon^2}{7.256 \delta} } + \sqrt{ \frac{\epsilon^2}{7.256 \delta} - 1 } \right) - \rmin \\
 \log \left(  \frac{\epsilon }{\delta} \cdot 1.268  \sinh \rmin  \right) - \rmin
& \geq \arccosh \frac{\epsilon}{\sqrt{7.256 \, \delta}} - \rmin.
\end{align*}
Thus, by \reflem{LowerBoundDeep}, 
the desired lower bound~\refeqn{DesiredBoundTubes} holds. 
\end{proof}

\begin{remark}
\label{Rem:LargeEpsilon}
In the above proof, the constant $\rmin = 0.0424$  is (a slight overestimate for) the maximum of the function $j(\delta,\epsilon)$ from \reflem{ComparisonFunction}. The maximum of $j(\delta, \epsilon)$ on its domain is attained when $\epsilon$ takes the maximal value $0.3$, and is the primary reason why the additive constant $-0.0424$ in the statement of \refthm{EffectiveDistTubes} depends on the upper bound for $\epsilon$.
Meanwhile, the multiplicative constant $7.256$ in the statement of \refthm{EffectiveDistTubes}  is a slight overestimate for $\sqrt{4 \pi  / \sqrt{3}}$. This multiplicative constant comes from \refprop{TubeRadBound} and \reflem{LowerBoundShallow}, where  $\delta \leq 2.97$ suffices. In particular, it carries no hypotheses on $\epsilon$.

Consequently, there is a generalization of \refthm{EffectiveDistTubes} that holds for larger values of $\epsilon$.
\end{remark}

\begin{theorem}
\label{Thm:EffectiveDistGeneral}
Fix a positive constant $ \epsilon_{\max} \leq 1.475$, and suppose that $0 < \delta < \epsilon \leq \epsilon_{\max}$.   Let $j(\delta,\epsilon)$ be the function of \reflem{ComparisonFunction}, and let $j_{\max}$ be the maximal value of $j(\delta, \epsilon_{\max})$ on the interval $0 \leq \delta \leq \epsilon_{\max}^2 / 7.256$. Let $\rmin = \arcsinh(j_{\max})$.

Then, for every hyperbolic solid torus $N = N_{\alpha, \lambda, \tau}$, where  $0 < \alpha \leq 2\pi$ and $\lambda \leq \delta$,
we have
\[
\max \left\{ \frac{\epsilon - \delta}{2}, \, 
\arccosh \frac{\epsilon}{\sqrt{7.256 \, \delta}} - \rmin \right\}
\leq
d_{\alpha, \lambda, \tau}(\delta,\epsilon)
\leq
\arccosh \sqrt{ \frac{\cosh \epsilon - 1}{\cosh \delta - 1}  }.
\]
\end{theorem}

\begin{proof}
The proof is nearly identical to the above proof of \refthm{EffectiveDistTubes}. We substitute $\epsilon_{\max}$ in place of $0.3$ and the new definition  $\rmin = \arcsinh(j_{\max})$ in place of $0.0424$. In the non-trivial case where $\delta \leq \epsilon^2/7.256$, the hypothesis $\epsilon \leq 1.475$  implies $\delta \leq 0.3$, which means \reflem{LowerBoundDeep} still applies. The final computation goes through verbatim.
\end{proof}

Finally, we observe that the hypotheses of \refthm{EffectiveDistGeneral} can be loosened even further, to $\epsilon_{\max} \leq \sqrt{7.256 \cdot 2.97} = 4.642\ldots,$ at the cost of a more complicated adjustment to the value of $\rmin$. Given $\delta_{\max} = \epsilon_{\max}^2/7.256$,  we have to adjust the statement of \reflem{LowerBoundDeep} to work for all $\delta$ in the range $0 < \delta \leq \delta_{\max}$. This will cause the multiplicative constant $1.268$ to change, which will cause the function $j(\delta,\epsilon)$ of \reflem{ComparisonFunction} to change as well. The maximum of the adjusted function will then determine the needed value of $\rmin$. 
  We expect the constants of \refthm{EffectiveDistGeneral} to suffice for our applications.

\appendix

\section{Hyperbolic trigonometry}
\label{Sec:Trig}

%
%
%

This appendix records several easy lemmas involving hyperbolic sines and cosines.

\begin{lemma}
\label{Lem:CoshMult}
Let $0 \leq r \leq s$. Then
\[
\cosh(s-r) \leq \frac{\cosh s}{\cosh r} ,
\]
with equality if and only if $r = 0$ or $r = s$.
\end{lemma}

\begin{proof}
Let $h = s-r$. Then 
\begin{align*}
\cosh(s) & = \cosh(r+h) \\
& = \cosh r \cosh h + \sinh r \sinh h \\
& \geq \cosh r \cosh h . 
\end{align*}
Note that we will have equality if and only if $r = 0$ or $h = 0$, as desired.
\end{proof}

\begin{lemma}
\label{Lem:SinhCoshGrowth}
Let $0 < r \leq s$. Then
\[
\frac{\cosh s}{\cosh r} \leq e^{s-r} \leq \frac{\sinh s}{\sinh r},
\]
with equality if and only if $r=s$.
\end{lemma}

\begin{proof}
Again, let $h = s-r$. Then, as above,
\begin{align*}
\cosh(s) 
& = \cosh r \cosh h + \sinh r \sinh h \\
& \leq \cosh r \cosh h + \cosh r \sinh h \\
& = \cosh r \cdot e^h,
\end{align*}
proving the first inequality. Observe that equality holds if and only if $h = 0$.
The second inequality is proved similarly, using the sum formula for $\sinh(r+h)$.
\end{proof}

\begin{lemma}
\label{Lem:CoshTaylorApprox}
Suppose $x, y \geq 0$. Define 
\[ f(x) = 
\begin{cases}
(\cosh x - 1)/x^2, & x > 0 \\
1/2, & x = 0 \\
\end{cases}
\qquad
\text{and} \qquad 
h(x,y) = \frac {f(y)}{f(x)}.
\]
Then $f(x)$ strictly increasing in $x$, while $h(x,y)$ is increasing in $y$ and decreasing in $x$.
\end{lemma}

\begin{proof}
Expanding the Taylor series
\[
\cosh x= 1 + \frac{x^2}{2!} + \frac{x^4}{4!} +  \frac{x^6}{6!} + \ldots
\]
gives
\[
f(x)=  \frac{x^0}{2!} + \frac{x^2}{4!} +  \frac{x^4}{6!}+  \ldots,
\]
which is clearly increasing in $x$. The monotonicity of $h$ is now immediate.
\end{proof}

%

%

\bibliographystyle{amsplain}
\bibliography{../biblio}

\end{document}

%% file: Example.pdf_tex
\begingroup%
  \makeatletter%
  \providecommand\color[2][]{%
    \errmessage{(Inkscape) Color is used for the text in Inkscape, but the package 'color.sty' is not loaded}%
    \renewcommand\color[2][]{}%
  }%
  \providecommand\transparent[1]{%
    \errmessage{(Inkscape) Transparency is used (non-zero) for the text in Inkscape, but the package 'transparent.sty' is not loaded}%
    \renewcommand\transparent[1]{}%
  }%
  \providecommand\rotatebox[2]{#2}%
  \ifx\svgwidth\undefined%
    \setlength{\unitlength}{97.59999847bp}%
    \ifx\svgscale\undefined%
      \relax%
    \else%
      \setlength{\unitlength}{\unitlength * \real{\svgscale}}%
    \fi%
  \else%
    \setlength{\unitlength}{\svgwidth}%
  \fi%
  \global\let\svgwidth\undefined%
  \global\let\svgscale\undefined%
  \makeatother%
  \begin{picture}(1,1.04106735)%
    \put(0,0){\includegraphics[width=\unitlength,page=1]{Example.pdf}}%
    \put(0.56417999,0.47731409){\color[rgb]{1,0,0}\makebox(0,0)[lb]{\smash{$\varphi$}}}%
    \put(0.26031884,0.71079558){\color[rgb]{1,0,0}\makebox(0,0)[lb]{\smash{$\varphi^2$}}}%
    \put(0.04837951,0.82008853){\color[rgb]{0,0,1}\makebox(0,0)[lb]{\smash{$\varphi^2$}}}%
    \put(0.79931568,0.61559724){\color[rgb]{0,0,1}\makebox(0,0)[lb]{\smash{$\varphi$}}}%
  \end{picture}%
\endgroup%

%% file: EffectiveDistTubes.bbl
\providecommand{\bysame}{\leavevmode\hbox to3em{\hrulefill}\thinspace}
\providecommand{\MR}{\relax\ifhmode\unskip\space\fi MR }
\providecommand{\MRhref}[2]{%
  \href{http://www.ams.org/mathscinet-getitem?mr=#1}{#2}
}
\providecommand{\href}[2]{#2}
\begin{thebibliography}{10}

\bibitem{ATW:DistanceFormula}
Tarik Aougab, Samuel~J. Taylor, and Richard~C.H. Webb, \emph{Effective
  {Masur--Minsky} distance formulas and applications to hyperbolic
  $3$--manifolds}, \url{https://math.temple.edu/~samuel.taylor/Main.pdf}.

\bibitem{bowditch:elc}
Brian~H. Bowditch, \emph{The ending lamination theorem},
  \url{http://homepages.warwick.ac.uk/~masgak/ papers/elt.pdf}.

\bibitem{brock:fibered}
Jeffrey~F. Brock, \emph{Weil--{P}etersson translation distance and volumes of
  mapping tori}, Comm. Anal. Geom. \textbf{11} (2003), no.~5, 987--999.
  \MR{MR2032506 (2004k:32018)}

\bibitem{brock:quasifuchsian}
\bysame, \emph{The {W}eil-{P}etersson metric and volumes of 3-dimensional
  hyperbolic convex cores}, J. Amer. Math. Soc. \textbf{16} (2003), no.~3,
  495--535 (electronic). \MR{MR1969203 (2004c:32027)}

\bibitem{brock-bromberg:density}
Jeffrey~F. Brock and Kenneth~W. Bromberg, \emph{On the density of geometrically
  finite {K}leinian groups}, Acta Math. \textbf{192} (2004), no.~1, 33--93.
  \MR{MR2079598 (2005e:57046)}

\bibitem{brock-bromberg:inflexibility}
\bysame, \emph{Inflexibility, {W}eil-{P}eterson distance, and volumes of
  fibered 3-manifolds}, Math. Res. Lett. \textbf{23} (2016), no.~3, 649--674.
  \MR{3533189}

\bibitem{brock-canary-minsky:elc}
Jeffrey~F. Brock, Richard~D. Canary, and Yair~N. Minsky, \emph{{The
  classification of Kleinian surface groups, II: The Ending Lamination
  Conjecture}}, 2004.

\bibitem{brooks-matelski}
Robert Brooks and J.~Peter Matelski, \emph{Collars in {K}leinian groups}, Duke
  Math. J. \textbf{49} (1982), no.~1, 163--182. \MR{650375 (83f:30039)}

\bibitem{CaoGehringMartin}
Chun Cao, Frederick~W. Gehring, and Gaven~J. Martin, \emph{Lattice constants
  and a lemma of {Z}agier}, Lipa's legacy ({N}ew {Y}ork, 1995), Contemp. Math.,
  vol. 211, Amer. Math. Soc., Providence, RI, 1997, pp.~107--120. \MR{1476983}

\bibitem{cfp:tunnels}
Daryl Cooper, David Futer, and Jessica~S Purcell, \emph{Dehn filling and the
  geometry of unknotting tunnels}, Geom. Topol. \textbf{17} (2013), no.~3,
  1815--1876. \MR{3073937}

\bibitem{CullerShalen:VolumeBetti2}
Marc Culler and Peter~B. Shalen, \emph{The volume of a hyperbolic
  {$3$}-manifold with {B}etti number {$2$}}, Proc. Amer. Math. Soc.
  \textbf{120} (1994), no.~4, 1281--1288. \MR{1205485}

\bibitem{culler-shalen:margulis}
\bysame, \emph{Margulis numbers for {H}aken manifolds}, Israel J. Math.
  \textbf{190} (2012), 445--475. \MR{2956250}

\bibitem{FPS:EffectiveBilipschitz}
David Futer, Jessica~S. Purcell, and Saul Schleimer, \emph{Effective
  bilipschitz bounds on drilling and filling}, In preparation.

\bibitem{futer-schleimer:cusp-geometry}
David Futer and Saul Schleimer, \emph{Cusp geometry of fibered 3-manifolds},
  Amer. J. Math. \textbf{136} (2014), no.~2, 309--356. \MR{3188063}

\bibitem{GabaiMeyerhoffMilley:VolTubes}
David Gabai, G.~Robert Meyerhoff, and Peter Milley, \emph{Volumes of tubes in
  hyperbolic 3-manifolds}, J. Differential Geom. \textbf{57} (2001), no.~1,
  23--46. \MR{1871490}

\bibitem{HPW:SlimUnicorns}
Sebastian Hensel, Piotr Przytycki, and Richard~C.H. Webb, \emph{$1$--slim
  triangles and uniform hyperbolicity for arc graphs and curve graphs}, J. Eur.
  Math. Soc. (JEMS) \textbf{17} (2015), no.~4, 755--762.

\bibitem{hk:univ}
Craig~D. Hodgson and Steven~P. Kerckhoff, \emph{Universal bounds for hyperbolic
  {D}ehn surgery}, Ann. of Math. (2) \textbf{162} (2005), no.~1, 367--421.
  \MR{2178964 (2006g:57031)}

\bibitem{hk:shape}
\bysame, \emph{The shape of hyperbolic {D}ehn surgery space}, Geom. Topol.
  \textbf{12} (2008), no.~2, 1033--1090. \MR{2403805 (2010b:57021)}

\bibitem{meyerhoff}
Robert Meyerhoff, \emph{A lower bound for the volume of hyperbolic
  {$3$}-manifolds}, Canad. J. Math. \textbf{39} (1987), no.~5, 1038--1056.
  \MR{918586 (88k:57049)}

\bibitem{minsky:punctured-tori}
Yair~N. Minsky, \emph{The classification of punctured-torus groups}, Ann. of
  Math. (2) \textbf{149} (1999), no.~2, 559--626. \MR{MR1689341 (2000f:30028)}

\bibitem{minsky:models-bounds}
\bysame, \emph{{The classification of Kleinian surface groups, I: Models and
  bounds}}, Ann. of Math. (2) \textbf{171} (2010), 1--107.

\bibitem{namazi-souto:density}
Hossein Namazi and Juan Souto, \emph{Non-realizability and ending laminations:
  proof of the density conjecture}, Acta Math. \textbf{209} (2012), no.~2,
  323--395. \MR{3001608}

\bibitem{Ohshika05}
Ken'ichi Ohshika, \emph{Kleinian groups which are limits of geometrically
  finite groups}, Mem. Amer. Math. Soc. \textbf{177} (2005), no.~834, xii+116.
  \MR{2154090 (2006b:57021)}

\bibitem{shalen:margulis-numbers}
Peter~B. Shalen, \emph{A generic {M}argulis number for hyperbolic 3-manifolds},
  Topology and geometry in dimension three, Contemp. Math., vol. 560, Amer.
  Math. Soc., Providence, RI, 2011, pp.~103--109. \MR{2866926 (2012i:57036)}

\bibitem{Shalen:SmallOptimalMargulis}
\bysame, \emph{Small optimal {M}argulis numbers force upper volume bounds},
  Trans. Amer. Math. Soc. \textbf{365} (2013), no.~2, 973--999. \MR{2995380}

\bibitem{Susskind}
Perry Susskind, \emph{The {M}argulis region and continued fractions}, Complex
  manifolds and hyperbolic geometry ({G}uanajuato, 2001), Contemp. Math., vol.
  311, Amer. Math. Soc., Providence, RI, 2002, pp.~335--343. \MR{1940179}

\bibitem{thurston:notes}
William~P. Thurston, \emph{The geometry and topology of three-manifolds},
  Princeton Univ. Math. Dept. Notes, 1980, {\tt http://\allowbreak
  www.msri.org/\allowbreak gt3m/}.

\end{thebibliography}
